\documentclass[12pt]{article}
\usepackage{amsmath,amsfonts,amssymb,amsthm,mathrsfs,graphics,graphicx}

\newcommand{\I}{{\bf 1}}
\newtheorem{proposition}{Proposition}[section]
\newtheorem{theorem}[proposition]{Theorem}

\newtheorem{lemma}[proposition]{Lemma}
\newtheorem{remark}[proposition]{Remark}

\newcommand{\nc}{\newcommand}
\nc{\R}{{\mathbb R}}
\nc{\bS}{{\mathbb S}^{d-1}}
\nc{\N}{{\mathbb N}}
\nc{\C}{{\mathbb C}}
\nc{\Z}{{\mathbb Z}}
\nc{\BP}{\mathbb{P}}
\nc{\BE}{\mathbb{E}}
\nc{\BQ}{\mathbb{Q}}
\nc{\bN}{{\mathbf N}}
\nc{\BX}{{\mathbb X}}
\nc{\BY}{{\mathbb Y}}
\nc{\cB}{{\mathcal B}}
\nc{\cE}{{\mathcal E}}
\nc{\cL}{{\mathcal L}}
\nc{\cX}{{\mathcal X}}
\nc{\cY}{{\mathcal Y}}
\nc{\dint}{{\rm d}}
\nc{\D}{\Delta}
\nc{\g}{\gamma}
\nc{\cI}{\mathcal{I}}
\nc{\cZ}{\mathcal{Z}}
\nc{\cum}{\operatorname{cum}}
\nc{\sZ}{{\mathscr{Z}}}

\DeclareMathOperator{\BV}{\operatorname{Var}}

\DeclareMathOperator{\interior}{int}

\DeclareMathOperator{\RST}{\textup{RST}}

\numberwithin{equation}{section}

\begin{document}
\title{Central limit theorems for the radial spanning tree}
\date{}

\renewcommand{\thefootnote}{\fnsymbol{footnote}}

\author{Matthias Schulte\footnotemark[1]\;\, and Christoph Th\"ale\footnotemark[2]\,}

\footnotetext[1]{Institute of Stochastics, Karlsruhe Institute of Technology,
Germany, matthias.schulte@kit.edu}

\footnotetext[2]{Faculty of Mathematics, Ruhr University Bochum, Germany, christoph.thaele@rub.de}

\maketitle

\begin{abstract}
Consider a homogeneous Poisson point process in a compact convex set in $d$-dimensional Euclidean space which has interior points and contains the origin. The radial spanning tree is constructed by connecting each point of the Poisson point process with its nearest neighbour that is closer to the origin. For increasing intensity of the underlying Poisson point process the paper provides expectation and variance asymptotics as well as central limit theorems with rates of convergence for a class of edge functionals including the total edge length. \bigskip
\end{abstract}

\section{Introduction and results}

Random graphs for which the relative position of their vertices in space determines the presence of edges found considerable attention in the probability and combinatorics literature during the last decades,\ cf.\ \cite{BaBla,FranceschettiMester,Haenggi,Penrose}. Among the most popular models are the nearest-neighbour graph and the random geometric (or Gilbert) graph. Geometric random graphs with a tree structure have attracted particular interest. For example the minimal spanning tree has been studied intensively in stochastic optimization, cf.\ \cite{Steele,Yukich}. Bhatt and Roy \cite{BhattRoy} have proposed a model of a geometric random graph with a tree structure, the so-called minimal directed spanning tree, and in \cite{BaBo} another model has been introduced by Baccelli and Bordenave. This so-called radial spanning tree is also of interest in the area of communication networks as discussed in \cite{BaBo,Bordenave}.

To define the radial spaning tree formally, let $W\subset\R^d$, $d\geq 2$, be a compact convex set with $d$-dimensional Lebesgue measure $\lambda_d(W)>0$ which contains the origin $0$, and let $\eta_t$ be a Poisson point process in $W$ whose intensity measure is a multiple $t\geq 1$ of the Lebesgue measure restricted to $W$, see Section \ref{sec:Preparations} for more details. For a point $x\in\eta_t$ we denote by $n(x,\eta_t)$ the nearest neighbour of $x$ in $\eta_t\cup\{0\}$ which is closer to the origin than $x$, i.e., for which $\|n(x,\eta_t)\|\leq\|x\|$ and $\|x-n(x,\eta_t)\|\leq \|x-y\|$ for all $y\in\eta_t\cap (B^d(0,\|x\|)\setminus\{x\})$, where $\|\cdot\|$ stands for the Euclidean norm and $B^d(z,r)$ is the $d$-dimensional closed ball with centre $z\in\R^d$ and radius $r\geq 0$. In what follows, we call $n(x,\eta_t)$ the radial nearest neighbour of $x$. The radial spanning tree $\RST(\eta_t)$ with respect to $\eta_t$ rooted at the origin $0$ is the random tree in which each point $x\in\eta_t$ is connected by an edge to its radial nearest neighbour, see Figure \ref{fig1} and Figure \ref{fig3d} for pictures in dimensions $2$ and $3$.

In the original paper \cite{BaBo}, the radial spanning tree was constructed with respect to a stationary Poisson point process in $\R^d$, and properties dealing with individual edges or vertices such as edge lengths and degree distributions as well as the behaviour of semi-infinite paths were considered. Moreover, spatial averages of edge lengths within $W$ have been investigated, especially when $W$ is a ball with increasing radius. Semi-infinite paths and the number of infinite subtrees of the root were further studied by Baccelli, Coupier and Tran \cite{BacCoupierTran}, while Bordenave \cite{Bordenave} considered a closely related navigation problem. However, several questions related to cumulative functionals as the total edge length remained open. For example, the question of central limit theorem for the total edge length of the radial spanning tree within a convex observation window $W$ has been brought up by Penrose and Wade \cite{PenroseWade} and is still one of the prominent open problems. Note that for such a central limit theorem the set-up in which the intensity goes to infinity and the set $W$ is kept fixed is equivalent to the situation in which $W$ increases to $\R^d$ for fixed intensity. The main difficulty in proving a central limit theorem is that the usual existing techniques are (at least not directly) applicable, see \cite{PenroseWade} for a more detailed discussion. One reason for this is the lack of spatial homogeneity of the construction as a result of the observation that the geometry of the set of possible radial nearest neighbours of a given point changes with the distance of the point to the origin. The main contribution of our paper is a central limit theorem together with an optimal rate of convergence. Its proof relies on a recent Berry-Esseen bound for the normal approximation of Poisson functionals from Last, Peccati and Schulte \cite{LPS}, which because of its geometric flavour is particularly well suited for models arising in stochastic geometry. The major technical problem in order to prove a central limit theorem is to control appropriately the asymptotic behaviour of the variance. This sophisticated issue is settled in our text on the basis of a recent non-degeneracy condition also taken from \cite{LPS}.

Although a central limit theorem for the radial spanning tree is an open problem, we remark that central limit theorems for edge-length functionals of the minimal spanning tree of a random point set have been obtained by Kesten and Lee \cite{KestenLee} and later also by Chatterjee and Sen \cite{ChatterjeeSan}, Penrose \cite{PenroseCLT} and Penrose and Yukich \cite{PenroseYukich}. Moreover, Penrose and Wade \cite{PenroseWadeCLT} have shown a central limit theorem for the total edge length of the minimal directed spanning tree.

\medspace

\begin{figure}[t]
\begin{center}
\includegraphics[width=0.45\columnwidth]{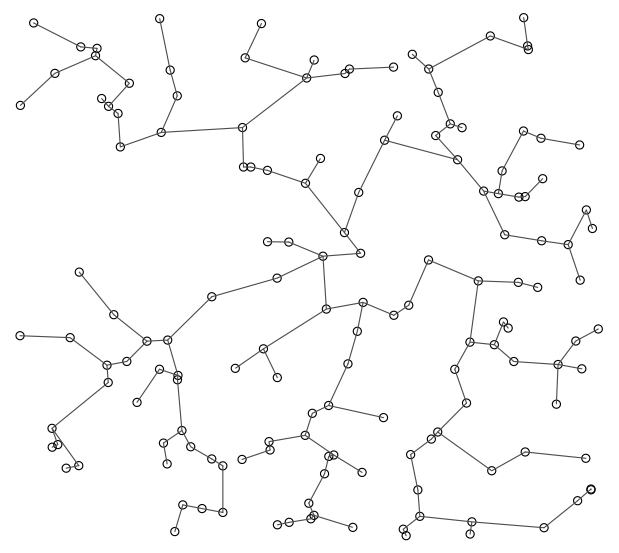}
\includegraphics[width=0.45\columnwidth]{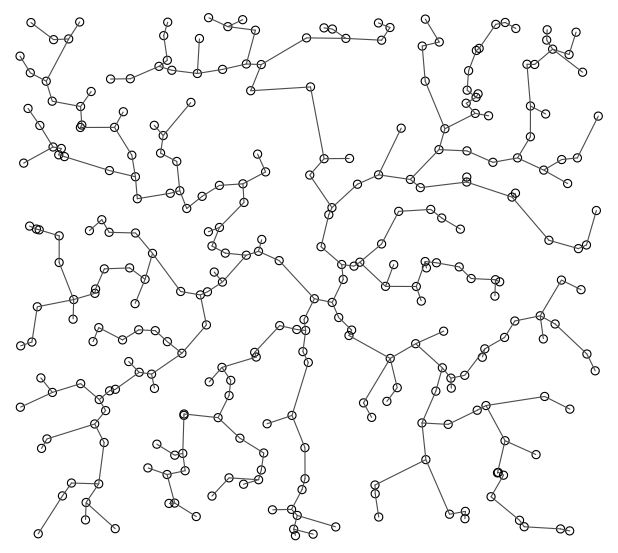}
\end{center}
\caption{\small Simulations of radial spanning trees in the unit square with $t=50$ (left) and $t=300$ (right). They have been produced with the freely available \texttt{R}-package \texttt{spatgraphs}.}\label{fig1}
\end{figure}

\begin{figure}[t]
\begin{center}
\includegraphics[width=0.6\columnwidth]{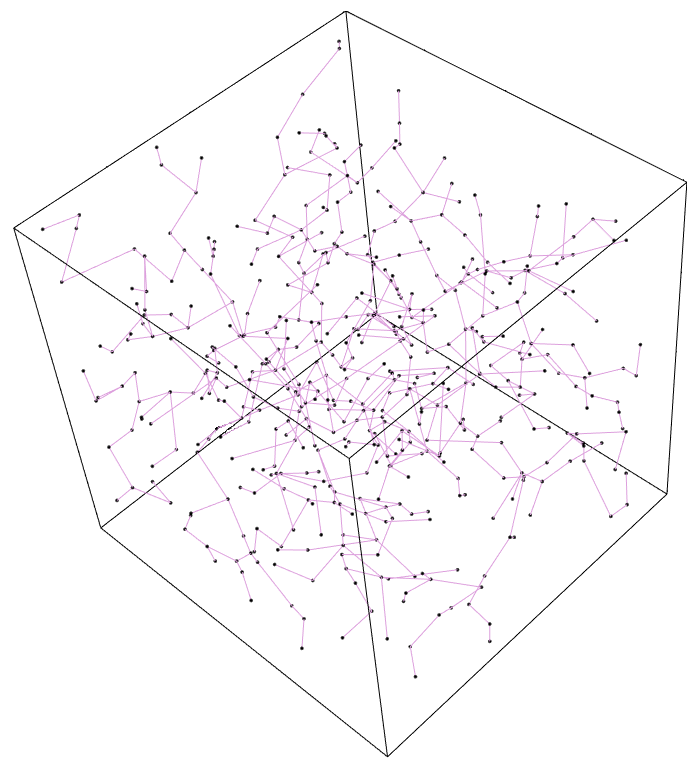}
\end{center}
\caption{\small A simulation of a radial spanning tree in the unit cube with $t=500$. It has been produced with the freely available \texttt{R}-package \texttt{spatgraphs}.}\label{fig3d}
\end{figure}

In order to state our main results formally, we use the notation $\ell(x,\eta_t):=\|x-n(x,\eta_t)\|$ and define the edge-length functionals
\begin{equation}\label{eq:LtaDef}
\cL_t^{(a)} := \sum_{x\in\eta_t} \ell(x,\eta_t)^a\,, \qquad a\geq 0\,, \quad t\geq 1\,,
\end{equation}
of $\RST(\eta_t)$ (the assumption that $a\geq 0$ is discussed in Remark \ref{rem:a>=0} below). Note that $\cL_t^{(0)}$ is just the number of vertices of $\RST(\eta_t)$, while $\cL_t^{(1)}$ is its total edge length. Our first result provides expectation and variance asymptotics for $\cL_t^{(a)}$. To state it, denote by $\kappa_d$ the volume of the $d$-dimensional unit ball and by $\Gamma(\,\cdot\,)$ the usual Gamma function. An explicit representation of the constant $v_a$ in the next theorem will be derived in Lemma \ref{lem:ExistenceVarianceAsymptotics} below.

\begin{theorem}\label{thm:Variance}
Let $a\geq 0$. Then
\begin{equation}\label{eq:ExpectationAsymptotics}
\lim_{t\to\infty}t^{a/d-1}\,\BE[\cL_t^{(a)}] = \left({2\over\kappa_d}\right)^{a/d}\Gamma\left(1+{a\over d}\right)\,\lambda_d(W)
\end{equation}
and there exists a constant $v_a\in(0,\infty)$ only depending on $a$ and $d$ such that
\begin{equation}\label{eq:VarianceAsymptotics}
\lim_{t\to\infty}t^{2a/d-1}\,\BV[\cL_t^{(a)}] = v_a\,\lambda_d(W)\,.
\end{equation}
\end{theorem}

We turn now to the central limit theorem. Our next result in particular ensures that, after suitable centering and rescaling, the functionals $\cL_t^{(a)}$ converge in distribution to a standard Gaussian random variable, as $t\to\infty$.

\begin{theorem}\label{thm:CLT}
Let $a\geq 0$ and let $Z$ be a standard Gaussian random variable. Then there is a constant $C\in(0,\infty)$ only depending on $W$, the parameter $a$ and the space dimension $d$ such that
$$
\sup_{s\in\R}\Big| \BP\Bigg(\frac{\cL_t^{(a)}-\BE[\cL_t^{(a)}]}{\sqrt{\BV[\cL_t^{(a)}]}}\leq s\Bigg) - \BP(Z \leq s)\Big| \leq C\, t^{-1/2}\,, \qquad t\geq 1\,.
$$
\end{theorem}

For $a=1$ and $a=0$ Theorem \ref{thm:CLT} says that the total edge length and the number of edges satisfy a central limit theorem, as $t\to\infty$. While for $a>0$ the result is non-trivial, a central limit theorem in the case $a=0$ is immediate from the following observations. Namely, there is a canonical one-to-one correspondence between the points of $\eta_t$ and the edges in the radial spanning tree. Moreover, the number of points of $\eta_t$ is a Poisson-distributed random variable $\eta_t(W)$ with mean (and variance) $t\lambda_d(W)$, and it is well-known from the classical central limit theorem that -- under suitable normalization -- such a random variable is well approximated by a standard Gaussian random variable. Since for this situation the rate $t^{-1/2}$ is known to be optimal, the rate of convergence in Theorem \ref{thm:CLT} can in general not be improved.

\medskip

The radial spanning tree is closely related to another geometric random graph, namely the directed spanning forest, which has also been introduced in \cite{BaBo}. The directed spanning forest $\operatorname{DSF}(\eta_t)$ with respect to a direction $e\in\mathbb{S}^{d-1}$ is constructed in the following way. For $x\in\R^d$ let $H_{x,e}$ be the half-space $H_{x,e}:=\{y\in\R^d: \langle e, y-x\rangle\leq 0\}$. We now take the points of $\eta_t$ as vertices of ${\rm DST}(\eta_t)$ and connect each point $x\in\eta_t$ with its closest neighbour $\hat{n}(x,\eta_t)$ in $\eta_t\cap (H_{x,e}\setminus\{x\})$. If there is no such point, we put $\hat{n}(x,\eta_t)=\emptyset$. This means that we look for the neighbour of a vertex of the directed spanning forest always in the same direction, whereas this direction changes according to the the relative position to the origin in case of the radial spanning tree. The directed spanning forest can be regarded as a local approximation of the radial spanning tree at distances far away from the origin.

As for the radial spanning tree we define $\ell_e(x,\eta_t):=\|x-\hat{n}(x,\eta_t)\|$ if $\hat{n}(x,\eta_t)\neq\emptyset$ and $\ell_e(x,\eta_t):=0$ if $\hat{n}(x,\eta_t)=\emptyset$, and
$$
\widehat{\mathcal{L}}_{t}^{(a)} := \sum_{x\in\eta_t} \ell_e(x,\eta_t)^a\,, \qquad a\geq 0\,, \quad t\geq 1\,.
$$
In order to avoid boundary effects it is convenient to replace the Poisson point process $\eta_t$ in $W$ by a unit-intensity stationary Poisson point process $\eta$ in $\R^d$. For this set-up we introduce the functionals
\begin{equation}\label{eq:widehatLW}
\widehat{\mathcal{L}}_W^{(a)} := \sum_{x\in\eta\cap W} \ell_e(x,\eta)^a\,, \quad a\geq 0\,, \qquad t\geq 1\,.
\end{equation}

Let us recall that a forest is a family of one or more than one disjoint trees, while a tree is an undirected and simple graph without cycles. Although with strictly positive probability $\operatorname{DSF}(\eta_t)$ is a union of more than one disjoint trees, Coupier and Tran have shown in \cite{CoupierTran} that in the directed spanning forest $\operatorname{DSF}(\eta)$ of a stationary Poisson point process $\eta$ \textit{in the plane} any two paths eventually coalesce, implying that $\operatorname{DSF}(\eta)$ is almost surely a tree. Resembling the intersection behaviour of Brownian motions, where the intersection of any finite number of independent Brownian motions in $\R^d$ with different starting points is almost surely non-empty if and only if $d=2$, cf.\ \cite[Theorem 9.3 (b)]{MP}, it seems that a similar property is no more true for dimensions $d\geq 3$. Note that in contrast to these dimension-sensitive results, our expectation and variance asymptotics and our central limit theorem hold for any space dimension $d\geq 2$. 

A key idea of the proof of Theorem \ref{thm:Variance} is to show that, for $a\geq 0$,
$$
\lim_{t\to\infty} t^{a/d-1}\BE[\mathcal{L}_t^{(a)}]=\BE[\widehat{\mathcal{L}}_W^{(a)}]
$$
and
\begin{equation}\label{eq:EVIntro}
\lim_{t\to\infty} t^{2a/d-1}\BV[\mathcal{L}_t^{(a)}]=\lim_{r\to\infty}\frac{\BV[\widehat{\mathcal{L}}_{B^d(0,r)}^{(a)}]}{\kappa_d r^d} \lambda_d(W)\,.
\end{equation}
In other words this means that the expectation and the variance of $\cL_t^{(a)}$ can be approximated by those of $\widehat{\cL}_W^{(a)}$ or $\widehat{\cL}_{B^d(0,r)}^{(a)}$, which are much easier to study because of translation invariance. We shall use a recent non-degeneracy criterion from \cite{LPS} to show that the right-hand side in \eqref{eq:EVIntro} is bounded away from zero, which implies then the same property also for the edge-length functionals of the radial spanning tree. To prove Theorem \ref{thm:CLT} we use again recent findings from \cite{LPS}. These are reviewed in Section \ref{sec:Preparations} together with some other background material. In Section \ref{sec:Variance} we establish Theorem \ref{thm:Variance} before proving Theorem \ref{thm:CLT} in the final Section \ref{sec:ProofCLT}.

\section{Preliminaries}\label{sec:Preparations}

\paragraph{Notation.}  Let $\mathcal{B}(\R^d)$ be the Borel $\sigma$-field on $\R^d$ and let $\lambda_d$ be the Lebesgue measure on $\R^d$. For a set $A\in\mathcal{B}(\R^d)$ we denote by $\interior(A)$ the interior of $A$ and by $\lambda_d|_A$ the restriction of $\lambda_d$ to $A$. For $z\in\R^d$ and $r\geq 0$, $B^d(z,r)$ stands for the closed ball with centre $z$ and radius $r$, and $B^d$ is the $d$-dimensional closed unit ball, whose volume is given by $\kappa_d:=\lambda_d(B^d)$. By $\mathbb{S}^{d-1}$ we denote the unit sphere in $\R^d$.

\paragraph{Poisson point processes.} Let $\bN$ be the set of $\sigma$-finite counting measures on $\R^d$ and let it be equipped with the $\sigma$-field that is generated by all maps $\bN\ni \mu\mapsto\mu(A)$, $A\in\mathcal{B}(\R^d)$. For a $\sigma$-finite non-atomic measure $\mu$ on $\R^d$ a Poisson point process $\eta$ with intensity measure $\mu$ is a random element in $\bN$ such that
\begin{itemize}
\item for all $n\in\N$ and pairwise disjoint sets $A_1,\ldots,A_n\in\mathcal{B}(\R^d)$ the random variables $\eta(A_1),\hdots,\eta(A_n)$ are independent,
\item for $A\in\mathcal{B}(\R^d)$ with $\mu(A)\in(0,\infty)$, $\eta(A)$ is Poisson distributed with parameter $\mu(A)$.
\end{itemize}
We say that $\eta$ is stationary with intensity $t>0$ if $\mu=t\lambda_d$, implying that $\eta$ has the same distribution as the translated point process $\eta+z$ for all $z\in\R^d$. By abuse of terminology we also speak about $t$ as intensity if the measure $\mu$ has the form $t\lambda_d|_W$ for some possibly compact subset $W\subset\R^d$.

A Poisson point process $\eta$ can with probability one be represented as
$$
\eta = \sum_{n=1}^m \delta_{x_n}, \quad x_1,x_2,\hdots\in \R^d, \qquad m\in\N\cup\{0,\infty\}\,,
$$
where $\delta_x$ stands for the unit-mass Dirac measure concentrated at $x\in\R^d$. In particular, if $\mu(\R^d)<\infty$, $\eta$ can be written in form of the distributional identity
$$
\eta=\sum_{n=1}^{M} \delta_{X_n}
$$
with i.i.d.\ random points $(X_n)_{n\in\N}$ that are distributed according to $\mu(\cdot)/\mu(\R^d)$ and a Poisson distributed random variable $M$ with mean $\mu(\R^d)$ that is independent of $(X_n)_{n\in\N}$.

As usual in the theory of point processes, we may identify $\eta$ with its support and write, by slight abuse of notation, $x\in\eta$ whenever $x$ is charged by the random measure $\eta$. Moreover, we write $\eta\cap A$ for the restriction of $\eta$ to a subset $A\in\mathcal{B}(\R^d)$. Although we interpreted $\eta$ as a random set in the introduction, we prefer from now on to regard $\eta$ as a random measure.

To deal with functionals of $\eta$, the so-called multivariate Mecke formula will turn out to be useful for us. It states that, for $k\in\N$ and a non-negative measurable function $f: (\R^d)^k\times \bN\to\R$,
\begin{equation}\label{eq:MeckeMulti}
\begin{split}
\BE\Big[\sum_{(x_1,\hdots,x_k)\in\eta_{\neq}^k} &f(x_1,\hdots,x_k,\eta)\Big] \\
&= \int_{(\R^d)^k}\BE[f(x_1,\hdots,x_k,\eta+\delta_{x_1}+\hdots+\delta_{x_k})]\,\mu^k(\dint(x_1,\hdots,x_k))\,,
\end{split}
\end{equation}
where $\eta_{\neq}^k$ stands for the set of all $k$-tuples of distinct points of $\eta$ and the integration on the right-hand side is with respect to the $k$-fold product measure of $\mu$ (see Theorem 1.15 in \cite{SWNewPers}). It is a remarkable fact that \eqref{eq:MeckeMulti} for $k=1$ also characterizes the Poisson point process $\eta$.

\paragraph{Variance asymptotics and central limit theorems for Poisson functionals.}

As a Poisson functional we denote a random variable $F$ which only depends on a Poisson point process $\eta$. Every Poisson functional $F$ can be written as $F=f(\eta)$ almost surely with a measurable function $f:\bN\to\R$, called a representative of $F$. For a Poisson functional $F$ with representative $f$ and $z\in \R^d$, the first-order difference operator is defined by
$$
D_zF := f(\eta+\delta_z) - f(\eta)\,.
$$
In other words, $D_zF$ measures the effect on $F$ when a point $z$ is added to $\eta$. Moreover, for $z_1,z_2\in \R^d$, the second-order difference operator is given by
\begin{align*}
D_{z_1,z_2}^2F &:= D_{z_1}(D_{z_2}F) =f(\eta+\delta_{z_1}+\delta_{z_2})-f(\eta+\delta_{z_1})-f(\eta+\delta_{z_2})+f(\eta)\,.
\end{align*}
The difference operators play a crucial role in the following result, which is the key tool to establish our central limit theorem, see \cite[Proposition 1.3]{LPS}.

\begin{proposition}\label{prop:LPS}
Let $(F_t)_{t\geq 1}$ be a family of Poisson functionals depending on the Poisson point processes $(\eta_t)_{t\geq 1}$ with intensity measures $t\lambda_d|_W$ for $t\geq 1$, where $W\subset\R^d$ is a compact convex set with interior points. Suppose that there are constants $c_1,c_2\in(0,\infty)$ such that
\begin{equation}\label{eq:MomentsDifferenceOperator}
\BE[|D_zF_t|^5]\leq c_1 \quad \text{and}\quad\BE[|D_{z_1,z_2}^2F_t|^5]\leq c_2\,, \qquad z,z_1,z_2\in W\,, \quad t\geq 1\,.
\end{equation}
Further assume that $t^{-1}\BV[F_t]\geq v$, $t\geq t_0$, for some $v,t_0\in(0,\infty)$ and that
\begin{equation}\label{eq:SupDifferenceOperator}
\sup_{z_1\in W,\,t\geq 1} t\int_{W}\BP(D_{z_1,z_2}^2F_t\neq 0)^{1/20}\,\dint z_2<\infty\,,
\end{equation}
and let $Z$ be a standard Gaussian random variable. Then there is a constant $C\in(0,\infty)$ such that
\begin{equation}\label{eq:BoundLPS}
\sup_{s\in\R} \Big| \BP\Big(\frac{F_t-\BE[F_t]}{\sqrt{\BV[F_t]}}\leq s \Big) -\BP(Z\leq s) \Big| \leq C\, t^{-1/2}\,, \qquad t\geq 1\,.
\end{equation}
\end{proposition}

The proof of Proposition \ref{prop:LPS} given in \cite{LPS} is based on a combination of Stein's method for normal approximation, the Malliavin calculus of variations and recent findings concerning the Ornstein-Uhlenbeck semigroup on the Poisson space around the so-called Mehler formula. Clearly, \eqref{eq:BoundLPS} implies that $(F_t-\BE[F_t])/\sqrt{\BV[F_t]}$ converges in distribution to a standard Gaussian random variable, as $t\to\infty$, but also delivers a rate of convergence for the so-called Kolmogorov distance. A similar bound is also available for the Wasserstein distance, but in order to keep the result transparent, we restrict here to the more prominent and (as we think) natural Kolmogorov distance.

One problem to overcome when one wishes to apply Proposition \ref{prop:LPS} is to show that the liming variance of the family $(F_t)_{t\geq 1}$ of Poisson functionals is not degenerate in that $t^{-1}\BV[F_t]$ is uniformly bounded away from zero for sufficiently large $t$. As discussed in the introduction we will relate the variance of $\mathcal{L}_t^{(a)}$, as $t\to\infty$, to the variance of $\widehat{\mathcal{L}}_{B^d(0,r)}^{(a)}$, as $r\to\infty$. A key tool to show that $\BV[\widehat{\mathcal{L}}_{W}^{(a)}]\geq v_a \lambda_d(W)$ with a constant $v_a\in(0,\infty)$ is the following result, which is a version of Theorem 5.2 in \cite{LPS}.

\begin{proposition}\label{prop:LPSvariance}
Let $\eta$ be a stationary Poisson point process in $\R^d$ with intensity measure $\lambda_d$ and let $F$ be a square-integrable Poisson functional depending on $\eta$ with representative $f:\bN\to\R$. Assume that there exist $k\in\N$, sets $I_1,I_2\subset\{1,\hdots,k\}$ with $I_1\cup I_2=\{1,\hdots,k\}$, a constant $c\in(0,\infty)$ and bounded sets $W_0,A_1,\hdots,A_k\subset\R^d$ with strictly positive Lebesgue measure such that
$$
\big| \BE\big[ f(\eta+\delta_x+\sum_{i\in I_1}\delta_{x+y_i}) - f(\eta+\delta_x+\sum_{i\in I_2}\delta_{x+y_i}) \big] \big| \geq c
$$
for $x\in W_0$ and $y_i\in A_i$, $i\in\{1,\hdots,k\}$. Then, there is a constant $v\in(0,\infty)$ only depending on $c$, $k$ and $A_1,\hdots,A_k$ such that
$$
\BV[F]\geq v \, \lambda_d(W_0)\,.
$$
\end{proposition}

\begin{proof} We define $U=\{(x_1,\hdots,x_{k+1})\in(\R^d)^{k+1}: x_1\in W_0, x_{i+1}\in x_1+A_i, i\in\{1,\hdots,k\}\}$. It follows from Theorem 5.2 in \cite{LPS} that
\begin{equation}\label{eq:BoundVarianceLPS}
\BV[F] \geq \frac{c^2}{4^{k+2}(k+1)!} \min_{\emptyset \neq J\subset \{1,\hdots,k+1\}} \inf_{\substack{V\subset U,\\ \lambda_d^{k+1}(V)\geq \lambda_d^{k+1}(U)/2^{k+2}}} \lambda_d^{|J|}(\Pi_J(V))\,,
\end{equation}
where $\Pi_J(\cdot)$ stands for the projection onto the components whose indices belong to $J$ and $|J|$ denotes the cardinality of $J$. We have that
\begin{equation}\label{eq:ZwischenrechnungVarianz1}
\begin{split}
\lambda_d^{k+1}(U) & = \int_{W_0}\int_{(\R^d)^k} \I\{ x_{i+1}\in x_1+A_i, i\in\{1,\hdots,k\}\} \, \dint(x_2,\hdots,x_{k+1}) \, \dint x_1\\
& = \lambda_d(W_0) \prod_{i=1}^k \lambda_d(A_i)\,.
\end{split}
\end{equation}
Since $A_1,\hdots,A_k$ are bounded, there is a constant $R\in(0,\infty)$ such that, for $(x_1,\hdots,x_{k+1})\in U$,
$$
\|x_i-x_j\|\leq R\,, \quad i,j\in\{1,\hdots,k+1\}\,.
$$
Let $V\subset U$ be such that $\lambda_d^{k+1}(V)\geq \lambda_d^{k+1}(U)/2^{k+2}$ and fix $\emptyset\neq J\subset \{1,\hdots,k+1\}$. We write $x_J$ (resp.\ $x_{J^C}$) for the vector consisting of all variables from $x_1,\ldots,x_{k+1}$ with index in $J$ (resp.\ $J^C$). Then, we have that
\begin{align*}
 \lambda_d^{k+1}(V) & \leq \int_{(\R^d)^{k+1}} \I\{x_J\in\Pi_J(V)\} \I\{(x_J,x_{J^C})\in U\} \, \dint(x_1,\hdots,x_{k+1})\\
& \leq  (\kappa_d R^d)^{|J^C|} \int_{(\R^d)^{|J|}} \I\{ x_J\in\Pi_J(V)\} \, \dint x_J = (\kappa_d R^d)^{|J^C|}\,\lambda_d^{|J|}(\Pi_J(V))\,.
\end{align*}
Because of $\lambda_d^{k+1}(V)\geq \lambda_d^{k+1}(U)/2^{k+2}$ and \eqref{eq:ZwischenrechnungVarianz1}, this implies that
$$
\lambda_d^{|J|}(\Pi_{J}(V))\geq \frac{\lambda_d^{k+1}(V)}{(\kappa_d R^d)^{|J^C|}} \geq \frac{ \prod_{i=1}^k \lambda_d(A_i)}{ 2^{k+2}(\kappa_d R^d)^{|J^C|}} \, \lambda_d(W_0)\,.
$$
Together with \eqref{eq:BoundVarianceLPS}, this concludes the proof.
\end{proof}

\section{Proof of Theorem \ref{thm:Variance}}\label{sec:Variance}

We prepare the proof of Theorem \ref{thm:Variance} by collecting some properties of the Poisson functionals we are interested in. Recall that the edge-length functional $\mathcal{L}_t^{(a)}$ defined at \eqref{eq:LtaDef} depends on the Poisson point process $\eta_t$ with intensity measure $t\lambda_d|_W$ for $t\geq 1$. It has representative
$$
\mathcal{L}^{(a)}(\operatorname{RST}(\xi)) = \sum_{x\in\xi} \ell(x,\xi)^a\,, \qquad \xi\in\bN\,,
$$
where $\ell(x,\xi)=\|x-n(x,\xi)\|$ is the distance between $x$ and its radial nearest neighbour $n(x,\xi)$ in $\xi+\delta_0$. The functional $\ell(\,\cdot\,,\,\cdot\,)$ is monotone and homogeneous in the sense that
\begin{equation}\label{eq:monotone}
\ell(x,\xi_1+\xi_2+\delta_x)\leq \ell(x,\xi_1+\delta_x)\,, \qquad x\in\R^d\,,\quad  \xi_1,\xi_2\in\bN\,,
\end{equation}
and
\begin{equation}\label{eq:homogene}
\ell(sx,s\xi)=s \ell(x,\xi)\,, \qquad x\in\R^d\,, \quad \xi\in\bN\,, \quad s>0\,,
\end{equation}
where $s\xi$ is the counting measure $\sum_{x\in\xi}\delta_{sx}$ we obtain by multiplying each point of $\xi$ with $s$.

The random variable $\widehat{\mathcal{L}}^{(a)}_W$ defined at \eqref{eq:widehatLW} depends on a stationary Poisson point process $\eta$ with intensity one. A representative of $\widehat{\mathcal{L}}^{(a)}_W$ is given by
$$
\widehat{\mathcal{L}}^{(a)}_W(\operatorname{DSF}(\xi))=\sum_{x\in\xi\cap W} \ell_e(x,\xi)^a\,, \qquad \xi\in\bN\,,
$$
where $\ell_e(x,\xi)=\|x-\hat{n}(x,\xi)\|$ is the distance between $x$ and its closest neighbour in $(\xi-\delta_x)\cap H_{x,e}$. Similarly to $\ell(\,\cdot\,,\,\cdot\,)$, the functional $\ell_e(\,\cdot\,,\,\cdot\,)$ is monotone in that
\begin{equation}\label{eq:monotoneelle}
\ell_e(x,\xi_1+\xi_2+\delta_x)\leq \ell_e(x,\xi_1+\delta_x)\,, \qquad x\in\R^d, \quad \xi_1,\xi_2\in\bN\,.
\end{equation}

We make use of two auxiliary results in the proof of Theorem \ref{thm:Variance}, which play a crucial role in Section \ref{sec:ProofCLT} as well.

\begin{lemma}\label{lem:KonstanteAlphad}
There is a constant $\alpha_W\in(0,\infty)$ only depending on $W$ and $d$ such that
$$
\lambda_d(B^d(x,u)\cap B^d(0,\|x\|)\cap W)\geq\alpha_W\lambda_d(B^d(x,u))=\alpha_W\kappa_d u^d
$$
for all $x\in W$ and $0\leq u\leq \|x\|$.
\end{lemma}
\begin{proof}
Let $x\in W$ and $0\leq u\leq \|x\|$ be given and define $\hat{x}:=(1-u/(2\|x\|))x$. Since $\|x-\hat{x}\|=u/2$ and $\|\hat{x}\|=\|x\|-u/2$, $B^d(\hat{x},u/2)$ is contained in $B^d(x,u)$ and in $B^d(0,\|x\|)$ so that
\begin{align}\label{eq:InclusionSets}
B^d(x,u)\cap B^d(0,\|x\|)\cap W \supset B^d(\hat{x},u/2)\cap W\,.
\end{align}
It follows from the proof of Lemma 2.5 in \cite{LastPenrose2013} that there is a constant $\gamma_W\in(0,\infty)$ only depending on $W$ and the dimension $d$ such that
$$
\lambda_d(B^d(y,r)\cap W)\geq \gamma_W \kappa_d r^d
$$
for all $y\in W$ and $0\leq r \leq \max_{z_1,z_2\in W}\|z_1-z_2\|$. Together with \eqref{eq:InclusionSets}, we obtain that
$$
\lambda_d(B^d(x,u)\cap B^d(0,\|x\|)\cap W) \geq \lambda_d(B^d(\hat{x},u/2)\cap W)\geq 2^{-d}\gamma_W \kappa_d u^d\,,
$$
and the choice $\alpha_W:=2^{-d}\gamma_W$ concludes the proof.
\end{proof}

The previous lemma allows us to derive exponential tails for the distribution of $\ell(x,\eta_t+\delta_x)$ as well as bounds for the moments.

\begin{lemma}\label{lem:ExpTail}
For all $t\geq 1$, $x\in W$ and $u\geq 0$ one has that
\begin{equation}\label{eq:BoundTailell}
\BP(\ell(x,\eta_t+\delta_x)\geq u) \leq \exp(-t\alpha_W\kappa_d u^d)\,,
\end{equation}
where $\alpha_W$ is the constant from Lemma \ref{lem:KonstanteAlphad}. Furthermore, for each $a\geq0$ there is a constant $c_a\in(0,\infty)$ only depending on $a$, $d$ and $\alpha_W$ such that
\begin{equation}\label{eq:BoundMomentsell}
t^{a/d} \,\BE[\ell(x,\eta_t+\delta_x)^{a}]\leq c_a\,, \qquad x\in W\,, \quad t\geq 1\,.
\end{equation}
\end{lemma}
\begin{proof}
Since $\ell(x,\eta_t+\delta_x)$ is at most $\|x\|$, \eqref{eq:BoundTailell} is obviously true if $u>\|x\|$. For $0\leq u \leq \|x\|$ we have that
\begin{align*}
\BP(\ell(x,\eta_t+\delta_{x})\geq u) & = \BP(\eta_t(B^d(x,u)\cap B^d(0,\|x\|)\cap W)=0)\\
&=\exp(-t\lambda_d(B^d(x,u)\cap B^d(0,\|x\|)\cap W))\\
&\leq \exp(-t\alpha_W\kappa_du^d)\,,
\end{align*}
where we have used that $\eta_t$ is a Poisson point process and Lemma \ref{lem:KonstanteAlphad}. For fixed $a>0$ the previous inequality implies that
\begin{equation}\label{eq:XXXRRRTTT}
\begin{split}
t^{a/d}\, \BE[\ell(x,\eta_t+\delta_x)^{a}] & = \int_0^\infty  \BP(t^{a/d}\ell(x,\eta_t+\delta_x)^{a} \geq u) \, \dint u\\
& \leq \int_0^\infty \exp(-\alpha_W\kappa_d u^{d/a}) \, \dint u\\
& = \frac{a}{d (\alpha_W\kappa_d)^{a/d}}\int_0^\infty v^{a/d-1} \exp(-v) \, \dint v\\
& = \frac{1}{(\alpha_W \kappa_d)^{a/d}}\Gamma\Big(1+\frac{a}{d}\Big)=:c_a
\end{split}
\end{equation}
for all $t\geq 1$. Finally, if $a=0$, \eqref{eq:BoundMomentsell} is obviously true with $c_0=1$.
\end{proof}

\begin{remark}\label{rem:a>=0}\rm
In Theorem \ref{thm:Variance} and Theorem \ref{thm:CLT} we assume that $a\geq 0$. One reason for this is the first equality in \eqref{eq:XXXRRRTTT}, which is false in case that $a<0$. A similar relation is also used for the variance asymptotics studied below.
\end{remark}

The proof of Theorem \ref{thm:Variance} is divided into several steps. We start with the expectation asymptotics, generalizing thereby the approach in \cite{BaBo} from the planar case to higher dimensions.

\begin{proof}[Proof of \eqref{eq:ExpectationAsymptotics}]
First, consider the case $a=0$. Here, we have that $$t^{a/d-1}\,\BE[\cL_t^{(a)}]=t^{-1}\,\BE[\eta_t(W)]=\lambda_d(W)\,.$$ Next, suppose that $a>0$. We use the Mecke formula \eqref{eq:MeckeMulti} to see that
$$
t^{a/d-1}\,\BE[\cL_t^{(a)}] =\int_W t^{a/d}\,\BE[\ell(x,\eta_t+\delta_x)^a]\,\dint x\,.
$$
By computing the expectation in the integral, we obtain that
\begin{equation}\label{eq:ExpectAsymp1}
t^{a/d-1}\,\BE[\cL_t^{(a)}] = a\int_W \int_0^\infty u^{a-1}\BP(t^{1/d}\ell(x,\eta_t+\delta_{x})\geq u)\,\dint u \,\dint x\,,
\end{equation}
where
\begin{align}\label{eq:WESExpectAsymp}
\BP(t^{1/d}\ell(x,\eta_t+\delta_{x})\geq u) =& \exp(-t\lambda_d(B^d(x,t^{-1/d}u)\cap B^d(0,\|x\|)\cap W))\,, \quad u\geq 0\,,
\end{align}
since $\eta_t$ is a Poisson point process with intensity measure $t\lambda_d|_W$. Since, as a consequence of Lemma \ref{lem:ExpTail},
$$
\BP(t^{1/d}\ell(x,\eta_t+\delta_{x})\geq u) \leq \exp(-\alpha_W\kappa_du^d)
$$
and
$$
a \int_W \int_0^\infty u^{a-1} \exp\left(-\alpha_W\kappa_d u^{d}\right)\,\dint u \, \dint x<\infty\,,
$$
we can apply the dominated convergence theorem to \eqref{eq:ExpectAsymp1} and obtain that
$$
\lim_{t\to\infty}t^{a/d-1}\,\BE[\cL_t^{(a)}] = a \int_W\int_0^\infty u^{a-1} \lim_{t\to\infty}\BP(t^{1/d}\ell(x,\eta_t+\delta_{x})\geq u)\,\dint u\,\dint x
$$
with the probability $\BP(t^{1/d}\ell(x,\eta_t+\delta_{x})\geq u)$ given by \eqref{eq:WESExpectAsymp}. For all $x\in\interior{(W)}$ we have that
$$
\lim_{t\to\infty} t\lambda_d(B^d(x,t^{-1/d}u)\cap B^d(0,\|x\|)\cap W)=\frac{\kappa_d}{2}u^d
$$
and, consequently,
$$
\lim_{t\to\infty} \BP(t^{1/d}\ell(x,\eta_t+\delta_{x})\geq u) = \exp\Big(-{1\over 2}\kappa_du^{d}\Big)\,.
$$
Summarizing, we find that
\begin{align*}
\lim_{t\to\infty}t^{a/d-1}\,\BE[\cL_t^{(a)}] &= \lambda_d(W)\,a\,\int_0^\infty u^{a-1}\,\exp\Big(-{1\over 2}\kappa_du^{d}\Big)\,\dint u\\
& = \left({2\over\kappa_d}\right)^{a/d}\Gamma\Big(1+{a\over d}\Big)\,\lambda_d(W)\,.
\end{align*}
This completes the proof of \eqref{eq:ExpectationAsymptotics}.
\end{proof}


Recall from the definition of the edge-length functionals $\widehat{\mathcal{L}}_W^{(a)}$ for the directed spanning forest that $\ell_e(x,\eta+\delta_x)$ is the distance from $x$ to the closest point of the unit-intensity stationary Poisson point process $\eta$, which is contained in the half-space $H_{x,e}$. A computation similar to that in the proof of \eqref{eq:ExpectationAsymptotics} shows that, for $a>0$ and $x\in\R^d$,
\begin{equation}\label{eq:Expectationellea}
\begin{split}
\BE[\ell_e(x,\eta+\delta_x)^a] & = \BE[\ell_e(0,\eta+\delta_0)^a] = a \int_0^\infty u^{a-1}\BP(\ell(0,\eta+\delta_0)\geq u) \, \dint u\\
 & = a \int_0^\infty u^{a-1}\exp(-\kappa_d u^d/2) \, \dint u =  \left({2\over\kappa_d}\right)^{a/d}\Gamma\left(1+{a\over d}\right)\,,
\end{split}
\end{equation}
whence, by the Mecke formula \eqref{eq:MeckeMulti},
$$
\lim_{t\to\infty} t^{a/d-1}\, \BE[\mathcal{L}^{(a)}_t] = \BE[\widehat{\mathcal{L}}_W^{(a)}]\,.
$$

Our next goal is to establish the existence of the variance limit. Positivity of the limiting variance is postponed to Lemma \ref{lem:VarianceStrictlyPositive} below.

\begin{lemma}\label{lem:ExistenceVarianceAsymptotics}
For any $a\geq 0$ the limit $\lim\limits_{t\to\infty}t^{2a/d-1}\BV[\cL_t^{(a)}]$ exists and equals $v_a\,\lambda_d(W)$ with a constant $v_a\in[0,\infty)$ given by
\begin{align*}
v_a & =\int_{\R^d} \BE[\ell_e(0,\eta+\delta_0+\delta_z)^a  \ell_e(z,\eta+\delta_0+\delta_z)^a] - \BE[\ell_e(0,\eta+\delta_0)^a ] \, \BE[\ell_e(z,\eta+\delta_z)^a ] \, \dint z\\
& \qquad\qquad\qquad + \BE[\ell_e(0,\eta+\delta_0)^{2a}]\,,
\end{align*}
where $e\in\mathbb{S}^{d-1}$ is some fixed direction.
\end{lemma}

\begin{proof}
For $a=0$, $t^{2a/d-1}\BV[\cL_t^{(a)}]=t^{-1}\BV[\eta_t(W)]=t^{-1}\BE[\eta_t(W)]=\lambda_d(W)$ since $\eta_t(W)$ is a Poisson random variable with mean $t\lambda_d(W)$. Hence, $v_0=1$ and we can and will from now on restrict to the case $a>0$, where we re-write the variance of $\mathcal{L}_t^{(a)}$ as
$$
\BV[\mathcal{L}^{(a)}_t] = \BE\Big[\sum_{(x,y)\in\eta^2_{t,\neq}} \ell(x,\eta_t)^a \ell(y,\eta_t)^a\Big]-\Big(\BE\Big[\sum_{x\in\eta_t} \ell(x,\eta_t)^{a}\,\Big]\Big)^2+\BE\Big[\sum_{x\in\eta_t} \ell(x,\eta_t)^{2a}\Big]\,.
$$
Now, we use the multivariate Mecke equation \eqref{eq:MeckeMulti} to deduce that $$t^{2a/d-1}\BV[\cL_t^{(a)}] = T_1(t)+T_2(t)$$ with $T_1(t)$ and $T_2(t)$ given by
\begin{align*}
T_1(t) &:= t\int_W\int_W t^{2a/d} \, \BE[\ell(x,\eta_t+\delta_x+\delta_y)^a \, \ell(y,\eta_t+\delta_x+\delta_y)^a]\\
&\qquad\qquad\qquad\qquad -t^{2a/d}\,\BE[\ell(x,\eta_t+\delta_x)^a] \, \BE[\ell(y,\eta_t+\delta_y)^a]\,\dint y \, \dint x\,,\\
T_2(t) &:= \int_W t^{2a/d}\,\BE[\ell(x,\eta_t+\delta_x)^{2a}]\,\dint x\,.
\end{align*}
It follows from $T_2(t)=t^{2a/d-1}\BE[\mathcal{L}^{(2a)}_t]$, the expectation asymptotics \eqref{eq:ExpectationAsymptotics} and \eqref{eq:Expectationellea} that
\begin{equation}\label{eq:limit3}
\lim_{t\to\infty} T_2(t) = \lim_{t\to\infty} t^{2a/d-1}\BE[\mathcal{L}^{(2a)}_t] =  \BE[\ell_e(0,\eta+\delta_0)^{2a}] \lambda_d(W)\,.
\end{equation}
Using the substitution $y=x+t^{-1/d}z$, we re-write $T_{1}(t)$ as
\begin{equation}\label{eq:T1}
\begin{split}
T_{1}(t) & = \int_W\int_{\R^d} \I\{x+t^{-1/d}z\in W\}\\
& \hskip 1.5cm \times \Big(t^{2a/d} \, \BE[\ell(x,\eta_t+\delta_x+\delta_{x+t^{-1/d}z})^a \ell(x+t^{-1/d}z,\eta_t+\delta_x+\delta_{x+t^{-1/d}z})^a]\\
& \hskip 2.15cm -t^{2a/d} \, \BE[\ell(x,\eta_t+\delta_x)^a]\,\BE[\ell(x+t^{-1/d}z,\eta_t+\delta_{x+t^{-1/d}z})^a]\Big) \,\dint z \, \dint x\,.
\end{split}
\end{equation}
Let $A$ be the union of the events
$$
A_1=\{\eta_t(B^d(x,t^{-1/d}\|z\|/2)\cap B^d(0,\|x\|))=0\}
$$
and
$$
A_2=\{\eta_t(B^d(x+t^{-1/d}z,t^{-1/d}\|z\|/2)\cap B^d(0,\|x+t^{-1/d}z\|))=0\}\,.
$$
Since the expectation factorizes on the complement $A^C$ of $A$ by independence, we have that
\begin{align*}
& \BE[\ell(x,\eta_t+\delta_x+\delta_{x+t^{-1/d}z})^a \, \ell(x+t^{-1/d}z,\eta_t+\delta_x+\delta_{x+t^{-1/d}z})^a\,{\bf 1}_{A^C}]\\
&= \BE[\ell(x,\eta_t+\delta_x)^a\,{\bf 1}_{A_1^C}] \, \BE[\ell(x+t^{-1/d}z,\eta_t+\delta_{x+t^{-1/d}z})^a\,{\bf 1}_{A_2^C}]\,.
\end{align*}
It follows from Lemma \ref{lem:KonstanteAlphad} that
$$
\BP(A_1)\leq \exp(-\alpha_W\kappa_d\|z\|^d/2^d)\,, \qquad \BP(A_2)\leq \exp(-\alpha_W\kappa_d\|z\|^d/2^d)
$$
and hence
$$
\BP(A) \leq 2 \exp(-\alpha_W\kappa_d\|z\|^d/2^d)\,.
$$
Thus, using \eqref{eq:monotone}, repeatedly the Cauchy-Schwarz inequality and Lemma \ref{lem:ExpTail}, we see that the integrand of $T_{1}(t)$ is bounded in absolute value from above by
\begin{align*}
& t^{2a/d} \, \BE[\ell(x,\eta_t+\delta_x+\delta_{x+t^{-1/d}z})^a \, \ell(x+t^{-1/d}z,\eta_t+\delta_x+\delta_{x+t^{-1/d}z})^a\,{\bf 1}_{A}]\\
& \quad + t^{2a/d} \, \BE[\ell(x,\eta_t+\delta_x)^a\,{\bf 1}_{A_1}] \, \BE[\ell(x+t^{-1/d}z,\eta_t+\delta_{x+t^{-1/d}z})^a]\\
& \quad + t^{2a/d} \, \BE[\ell(x,\eta_t+\delta_x)^a] \, \BE[\ell(x+t^{-1/d}z,\eta_t+\delta_{x+t^{-1/d}z})^a\,{\bf 1}_{A_2}]\\
& \leq \BE[t^{4a/d}\ell(x,\eta_t+\delta_x)^{4a}]^{1/4} \, \BE[t^{4a/d}\ell(x+t^{-1/d}z,\eta_t+\delta_{x+t^{-1/d}z})^{4a}]^{1/4}\\
& \quad\quad \times\sqrt{2} \exp(-\alpha_W\kappa_d\|z\|^d/2^{d+1})\\
& \quad +2\BE[t^{2a/d}\ell(x,\eta_t+\delta_x)^{2a}]^{1/2} \, \BE[t^{2a/d}\ell(x+t^{-1/d}z,\eta_t+\delta_{x+t^{-1/d}z})^{2a}]^{1/2} \\
& \quad\quad \times\exp(-\alpha_W\kappa_d\|z\|^d/2^{d+1})\\
& \leq  (\sqrt{2c_{4a}}+2c_{2a})\exp(-\alpha_W \kappa_d \|z\|^d/2^{d+1})\,.
\end{align*}
As a consequence, the integrand in \eqref{eq:T1} is bounded by an integrable function. We can thus apply the dominated convergence theorem and obtain that
\begin{align*}
\lim_{t\to\infty} T_{1}(t) \!& =\!\!  \int_W\int_{\R^d} \lim_{t\to\infty} \I\{x+t^{-1/d}z\in W\}\\
& \hskip 2.1cm \times\Big(t^{2a/d}\,\BE[\ell(x,\eta_t+\delta_x+\delta_{x+t^{-1/d}z})^a \, \ell(x+t^{-1/d}z,\eta_t+\delta_x+\delta_{x+t^{-1/d}z})^a]\\
& \hskip 2.75cm -t^{2a/d}\,\BE[\ell(x,\eta_t+\delta_x)^a]\,\BE[\ell(x+t^{-1/d}z,\eta_t+\delta_{x+t^{-1/d}z})^a]\Big) \,\dint z \, \dint x\,.
\end{align*}
Using the homogeneity relation \eqref{eq:homogene} and the fact that $t^{1/d}\eta_t$ has the same distribution as the restriction $(t^{1/d}x+\eta)|_{t^{1/d}W}$ of the translated Poisson point process $t^{1/d}x+\eta$ to $t^{1/d}W$, we see that
\begin{align*}
&t^{2a/d}\,\BE[\ell(x,\eta_t+\delta_x+\delta_{x+t^{-1/d}z})^a \, \ell(x+t^{-1/d}z,\eta_t+\delta_x+\delta_{x+t^{-1/d}z})^a]\\
& \qquad -t^{2a/d}\,\BE[\ell(x,\eta_t+\delta_x)^a] \, \BE[\ell(x+t^{-1/d}z,\eta_t+\delta_{x+t^{-1/d}z})^a]\\
& = \BE[\ell(t^{1/d}x,t^{1/d}\eta_t+\delta_{t^{1/d}x}+\delta_{t^{1/d}x+z})^a \, \ell(t^{1/d}x+z,t^{1/d}\eta_t+\delta_{t^{1/d}x}+\delta_{t^{1/d}x+z})^a]\\
& \qquad - \BE[\ell(t^{1/d}x,t^{1/d}\eta_t+\delta_{t^{1/d}x})^a] \, \BE[\ell(t^{1/d}x+z,t^{1/d}\eta_t+\delta_{t^{1/d}x+z})^a]\\
& = \BE[\ell(t^{1/d}x,(t^{1/d}x+\eta)|_{t^{1/d}W}+\delta_{t^{1/d}x}+\delta_{t^{1/d}x+z})^a\\
& \qquad\qquad\qquad \times\ell(t^{1/d}x+z,(t^{1/d}x+\eta)|_{t^{1/d}W}+\delta_{t^{1/d}x}+\delta_{t^{1/d}x+z})^a]\\
& \qquad - \BE[\ell(t^{1/d}x,(t^{1/d}x+\eta)|_{t^{1/d}W}+\delta_{t^{1/d}x})^a]\,\BE[\ell(t^{1/d}x+z,(t^{1/d}x+\eta)|_{t^{1/d}W}+\delta_{t^{1/d}x+z})^a]\,.
\end{align*}
It is not hard to verify that, for all $x\in\interior{(W)}$ with $x\neq0$ and $z\in\R^d$, the relations
\begin{align}
& \lim_{t\to\infty} \ell(t^{1/d}x,(t^{1/d}x+\eta)|_{t^{1/d}W}+\delta_{t^{1/d}x}+\delta_{t^{1/d}x+z}) = \ell_{x/\|x\|}(0,\eta+\delta_{0}+\delta_{z}) \label{eq:limit1}\\
& \lim_{t\to\infty} \ell(t^{1/d}x+z,(t^{1/d}x+\eta)|_{t^{1/d}W}+\delta_{t^{1/d}x}+\delta_{t^{1/d}x+z}) = \ell_{x/\|x\|}(z,\eta+\delta_{0}+\delta_{z}) \\
& \lim_{t\to\infty} \ell(t^{1/d}x,(t^{1/d}x+\eta)|_{t^{1/d}W}+\delta_{t^{1/d}x}) = \ell_{x/\|x\|}(0,\eta+\delta_0) \\
& \lim_{t\to\infty} \ell(t^{1/d}x+z,(t^{1/d}x+\eta)|_{t^{1/d}W}+\delta_{t^{1/d}x+z}) = \ell_{x/\|x\|}(z,\eta+\delta_z)\label{eq:limit2}
\end{align}
hold with probability one. In order to apply the dominated convergence theorem again, we have to find integrable majorants for the terms on the left-hand sides of \eqref{eq:limit1}--\eqref{eq:limit2}. First note that by the monotonicity relation \eqref{eq:monotone} we have that
$$
\ell(t^{1/d}x,(t^{1/d}x+\eta)|_{t^{1/d}W}+\delta_{t^{1/d}x}+\delta_{t^{1/d}x+z})\leq \ell(t^{1/d}x,(t^{1/d}x+\eta)|_{t^{1/d}W}+\delta_{t^{1/d}x})
$$
and
\begin{align*}
&\ell(t^{1/d}x+z,(t^{1/d}x+\eta)|_{t^{1/d}W}+\delta_{t^{1/d}x}+\delta_{t^{1/d}x+z})\\
&\qquad\qquad  \leq \ell(t^{1/d}x+z,(t^{1/d}x+\eta)|_{t^{1/d}W}+\delta_{t^{1/d}x+z})\,.
\end{align*}
For each $x\in\interior{(W)}$ we can construct a cone $K_x$ such that, for $t\geq 1$, a point $y\in t^{1/d} x+K_x$ is also in $t^{1/d}W$ if $\|y-t^{1/d}x\|\leq t^{1/d}\|x\|/2$. This can be done by choosing a point $p_x$ on the line between $x$ and $0$, which is sufficiently close to the origin $0$. Since this will always be an interior point of $W$, we find a ball $B^d(p_x,r_x)$ of a certain radius $r_x>0$ and centred at $p_x$, which is completely contained in $W$ and generates $K_x$ together with $x$. That is, $K_x$ is the smallest cone with apex $x$ containing $B^d(p_x,r_x)$. Then,
\begin{equation}\label{eq:BoundStationary1}
\ell(t^{1/d}x,(t^{1/d}x+\eta)|_{t^{1/d}W}+\delta_{t^{1/d}x})\leq 2\min_{y\in\eta\cap K_x}\|y\|\,.
\end{equation}
If the opening angle of the cone $K_x$ is sufficiently small, we have that $\|y\|\leq \|t^{1/d}x+z\|$ if $y\in t^{1/d}x +K_x$ and $\|y-t^{1/d}x\|\geq 2\|z\|$. Thus, it holds that
\begin{equation}\label{eq:BoundStationary2}
\ell(t^{1/d}x+z,(t^{1/d}x+\eta)|_{t^{1/d}W}+\delta_{t^{1/d}x+z}) \leq 2 \min_{y\in\eta\cap K_x, \|y\|\geq 2\|z\|}\|y\|\,.
\end{equation}
It follows from the fact that $\eta$ is a stationary Poisson point process with intensity one that
$$
\BP(\min_{y\in\eta\cap K_x, \|y\|\geq 2\|z\|}\|y\|\geq u) = \exp\big(-\lambda_d(K_x\cap B^d(0,1))(u^d-2^d\|z\|^d)\big)
$$
for $u\geq 2\|z\|$. Hence, the right-hand sides in \eqref{eq:BoundStationary1} and \eqref{eq:BoundStationary2} are integrable majorants for the expressions on the left-hand sides in \eqref{eq:limit1}--\eqref{eq:limit2}. Thus, we obtain by the dominated convergence theorem that
\begin{align*}
& \lim_{t\to\infty} \I\{x+t^{-1/d}z\in W\} t^{2a/d}\Big(\BE[\ell(x,\eta_t+\delta_x+\delta_{x+t^{-1/d}z})^a \ell(x+t^{-1/d}z,\eta_t+\delta_x+\delta_{x+t^{-1/d}z})^a]\\
& \hskip 5.5cm -\BE[\ell(x,\eta_t+\delta_x)^a]\,\BE[\ell(x+t^{-1/d}z,\eta_t+\delta_{x+t^{-1/d}z})^a]\Big)\\
& = \BE[\ell_{x/\|x\|}(0,\eta+\delta_0+\delta_z)^a \, \ell_{x/\|x\|}(z,\eta+\delta_0+\delta_z)^a] \\
&\qquad\qquad\qquad\qquad- \BE[\ell_{x/\|x\|}(0,\eta+\delta_0)^a]\,\BE[\ell_{x/\|x\|}(z,\eta+\delta_{z})^a]\,.
\end{align*}
Together with the fact that the expectations in the previous term are independent of the choice of $x\in\R^d$, we see that
\begin{align*}
\lim_{t\to\infty}T_1(t) = \lambda_d(W) \int_{\R^d} & \BE[\ell_e(0,\eta+\delta_0+\delta_z)^a \, \ell_e(z,\eta+\delta_0+\delta_z)^a]\\
 &\qquad\qquad - \BE[\ell_e(0,\eta+\delta_0)^a]\,\BE[\ell_e(z,\eta+\delta_{z})^a] \, \dint z\,,
\end{align*}
which together with \eqref{eq:limit3} concludes the proof.
\end{proof}

After having established the existence of the variance limit in \eqref{eq:VarianceAsymptotics}, we need to show that it is bounded away from zero. As a first step, the next lemma relates the asymptotic variance of $\mathcal{L}_t^{(a)}$, as $t\to\infty$, to that of $\widehat{\mathcal{L}}_{B^d(0,r)}^{(a)}$, as $r\to\infty$. Recall \eqref{eq:widehatLW} for the definition of $\widehat{\mathcal{L}}_{B^d(0,r)}^{(a)}$ and note that $\widehat{\mathcal{L}}_{B^d(0,r)}^{(a)}$ depends on a direction $e\in\mathbb{S}^{d-1}$, which is suppressed in our notation.

\begin{lemma}\label{lem:VarianceDSF}
For $a\geq 0$,
$$
\lim\limits_{r\to\infty} \frac{\BV[\widehat{\mathcal{L}}_{B^d(0,r)}^{(a)}]}{\kappa_d r^d}=v_a\,,
$$
where $v_a$ is the constant from Lemma \ref{lem:ExistenceVarianceAsymptotics}.
\end{lemma}

\begin{proof}
For $a=0$, $\widehat{\mathcal{L}}_{B^d(0,r)}^{(a)}$ is a Poisson random variable with mean $\lambda_d(B^d(0,r))=\kappa_dr^d$ and the statement is thus satisfied with $v_0=1$. So, we may assume that $a>0$. It follows from the multivariate Mecke equation \eqref{eq:MeckeMulti} and the same argument as at the beginning of the proof of Lemma \ref{lem:ExistenceVarianceAsymptotics} that
$$
\BV[\widehat{\cL}_{B^d(0,r)}^{(a)}] = T_{1,r}+T_{2,r}
$$
with $T_{1,r}$ and $T_{2,r}$ given by
\begin{align*}
T_{1,r} &:= \int_{B^d(0,r)}\int_{\mathbb{R}^d} \I\{y\in B^d(0,r)\} \Big(\BE[\ell_e(x,\eta+\delta_x+\delta_y)^a \, \ell_e(y,\eta+\delta_x+\delta_y)^a]\\
&\qquad\qquad\qquad\qquad\qquad\qquad\qquad -\BE[\ell_e(x,\eta+\delta_x)^a]\,\BE[\ell_e(y,\eta+\delta_y)^a]\Big) \,\dint y \, \dint x\,,\\
T_{2,r} &:= \int_{B^d(0,r)} \BE[\ell_e(x,\eta+\delta_x)^{2a}]\,\dint x\,.
\end{align*}
By the translation invariance of $\ell_e$ and $\eta$ we have that
\begin{align*}
T_{1,r} &:= \int_{B^d(0,r)}\int_{\mathbb{R}^d} \I\{y-x\in B^d(-x,r)\}\\
& \hskip 2.5cm  \Big(\BE[\ell_e(0,\eta+\delta_0+\delta_{y-x})^a \, \ell_e(y-x,\eta+\delta_0+\delta_{y-x})^a]\\
& \hskip 2.75cm -\BE[\ell_e(0,\eta+\delta_{0})^a]\,\BE[\ell_e(y-x,\eta+\delta_{y-x})^a]\Big) \,\dint y \, \dint x \allowdisplaybreaks\\
&= \int_{B^d(0,r)}\int_{\mathbb{R}^d} \I\{y\in B^d(-x,r)\} \Big(\BE[\ell_e(0,\eta+\delta_0+\delta_y)^a \, \ell_e(y,\eta+\delta_0+\delta_y)^a]\\
& \hskip 6cm -\BE[\ell_e(0,\eta+\delta_0)^a]\,\BE[\ell_e(y,\eta+\delta_y)^a]\Big) \,\dint y \, \dint x \allowdisplaybreaks\\
&= r^d \int_{B^d(0,1)}\int_{\mathbb{R}^d} \I\{y\in B^d(-rx,r)\} \Big(\BE[\ell_e(0,\eta+\delta_0+\delta_y)^a \, \ell_e(y,\eta+\delta_0+\delta_y)^a]\\
& \hskip 6.5cm -\BE[\ell_e(0,\eta+\delta_0)^a]\,\BE[\ell_e(y,\eta+\delta_y)^a]\Big) \,\dint y \, \dint x\,.
\end{align*}
Let $y\in\R^d$ and let $A_1$ and $A_2$ be the events
$$A_1:=\{\eta(B^d(0,\|y\|/2))=0\}\qquad\text{and}\qquad A_2:=\{\eta(B^d(y,\|y\|/2))=0\}\,,$$
respectively. On $(A_1\cup A_2)^C$ the expectation factorizes so that
\begin{align*}
&\BE[\ell_e(0,\eta+\delta_0+\delta_y)^a \, \ell_e(y,\eta+\delta_0+\delta_y)^a \I_{(A_1\cup A_2)^C}] \\
&\qquad\qquad= \BE[\ell_e(0,\eta+\delta_0)^a \I_{A_1^C}]\,\BE[\ell_e(y,\eta+\delta_y)^a \I_{A_2^C}]\,.
\end{align*}
Together with the Cauchy-Schwarz inequality and \eqref{eq:monotoneelle}, this implies that
\begin{align*}
& \big|\BE[\ell_e(0,\eta+\delta_0+\delta_y)^a \, \ell_e(y,\eta+\delta_0+\delta_y)^a] -\BE[\ell_e(0,\eta+\delta_0)^a]\,\BE[\ell_e(y,\eta+\delta_y)^a]\big|\\
& \leq \BE[\ell_e(0,\eta+\delta_0+\delta_y)^a \, \ell_e(y,\eta+\delta_0+\delta_y)^a \I_{(A_1\cup A_2)}]\\
& \quad +\BE[\ell_e(0,\eta+\delta_0)^a \I_{A_1}]\,\BE[\ell_e(y,\eta+\delta_y)^a]+\BE[\ell_e(0,\eta+\delta_0)^a]\,\BE[\ell_e(y,\eta+\delta_y)^a \I_{A_2}]\\
& \leq \BE[\ell_e(0,\eta+\delta_0)^{4a}]^{1/4} \,  \BE[\ell_e(y,\eta+\delta_y)^{4a}]^{1/4} \, \BP(A_1\cup A_2)^{1/2}\\
& \quad +\BE[\ell_e(0,\eta+\delta_0)^{2a}]^{1/2} \, \BE[\ell_e(y,\eta+\delta_y)^{2a}]^{1/2} \, (\BP(A_1)^{1/2}+\BP(A_2)^{1/2})\,.
\end{align*}
Since $\BP(\ell_e(0,\eta+\delta_0)\geq u) = \exp(-\kappa_d u^d/2)$ for $u\geq 0$, all moments of $\ell_e(0,\eta+\delta_0)$ are finite. Moreover, we have that
$$
\BP(A_1)=\BP(A_2)=\exp(-\kappa_d \|y\|^d/2^d) \quad \text{ and } \quad \BP(A_1\cup A_2)\leq 2\exp(-\kappa_d \|y\|^d/2^d)\,.
$$
Hence, there is a constant $C_a\in(0,\infty)$ such that
\begin{align*}
& \big|\BE[\ell_e(0,\eta+\delta_0+\delta_y)^a \, \ell_e(y,\eta+\delta_0+\delta_y)^a] -\BE[\ell_e(0,\eta+\delta_0)^a]\,\BE[\ell_e(y,\eta+\delta_y)^a]\big|\\
&\qquad \leq C_a \exp(-\kappa_d \|y\|^d/2^{d+1})\,.
\end{align*}
This allows us to apply the dominated convergence theorem, which yields
\begin{align*}
\lim_{r\to\infty} \frac{T_{1,r}}{\kappa_d r^d} &= \int_{\R^d} \BE[\ell_e(0,\eta+\delta_0+\delta_y)^a  \ell_e(y,\eta+\delta_0+\delta_y)^a]\\
&\qquad\qquad\qquad - \BE[\ell_e(0,\eta+\delta_0)^a ] \, \BE[\ell_e(y,\eta+\delta_y)^a ] \, \dint y\,.
\end{align*}
On the other hand we have that
$$T_{2,r}=\kappa_d r^d\,\BE[\ell_e(0,\eta+\delta_0)^{2a}]\,.$$
This completes the proof.
\end{proof}

Lemma \ref{lem:ExistenceVarianceAsymptotics} and Lemma \ref{lem:VarianceDSF} show that the suitably normalized functionals $\cL_t^{(a)}$ and $\widehat{\cL}_{B^d(0,r)}^{(a)}$ have the same limiting variance $v_a$, as $t\to\infty$ or $r\to\infty$, respectively. To complete the proof of Theorem \ref{thm:Variance} it remains to show that $v_a$ is strictly positive. Our preceding observation shows that for this only an analysis of the directed spanning forest ${\rm DSF}(\eta)$ of a stationary Poisson point process $\eta$ in $\R^d$ with intensity $1$ is necessary. Compared to the radial spanning tree and as discussed already in the introduction, the advantage of this model is that its construction is homogeneous in space, as it is based on the notion of half-spaces. This makes it much easier to analyse. Recall the definition \eqref{eq:widehatLW} of $\widehat{\cL}_W^{(a)}$ and note that it depends on a previously chosen direction $e\in\mathbb{S}^{d-1}$.

\begin{lemma}\label{lem:VarianceStrictlyPositive}
For all $e\in\mathbb{S}^{d-1}$ and $a\geq 0$ one has that $\BV[\widehat{\cL}^{(a)}_W]\geq \hat{v}_{a} \lambda_d(W)$ with a constant $\hat{v}_a\in(0,\infty)$ only depending on $a$, $d$ and
$$
r_W:=\sup\{u\in(0,1]: \lambda_d(\{x\in W: B^d(x,\max\{2,\sqrt{d}\}u)\subset W\})\geq \lambda_d(W)/2\}\,.
$$
In particular, the constants $v_a$, $a\geq 0$, defined in Lemma \ref{lem:ExistenceVarianceAsymptotics} are strictly positive.
\end{lemma}
\begin{proof}
Since $v_0=1$ as observed above, we restrict from now on to the case $a>0$. Without loss of generality, we can assume that $e=(0,-1,0,\hdots,0)$. We define $W_0=\{x\in W: B^d(x,\max\{2,\sqrt{d}\}r_W)\subset W\}$. 

For technical reasons, we have to distinguish two cases and start with the situation that $a\geq 1$. Let us define
$$
z_1:=(-1/2,-\sqrt{3}/2,0,\hdots,0)\,,\ z_2:=(1/2,-\sqrt{3}/2,0,\hdots,0)\,,\ z_3:=(0,-\sqrt{3}/2,0,\hdots,0)\,.
$$
The points $0$, $z_1$ and $z_2$ form an equilateral triangle with side-length one, and $z_3$ is the midpoint between $z_1$ and $z_2$, see Figure \ref{fig:Variance}. Because of $2 > 2 (1/2)^a + (\sqrt{3}/2)^a$ we have that
$$
\ell_e(z_1,\delta_{0}+\delta_{z_1}+\delta_{z_2})^a + \ell_e(z_2,\delta_{0}+\delta_{z_1}+\delta_{z_2})^a > c+\|z_1-z_3\|^a+ \|z_2-z_3\|^a + \|z_3\|^a
$$
with a constant $c\in(0,\infty)$. By continuity, there is a constant $r_0\in(0,r_W)$ such that
$$
\ell_e(y_1,\delta_{0}+\delta_{y_1}+\delta_{y_2})^a + \ell_e(y_2,\delta_{0}+\delta_{y_1}+\delta_{y_2})^a > \frac{c}{2}+\|y_1-y_3\|^a+ \|y_2-y_3\|^a + \|y_3\|^a
$$
for all $y_1\in B^d(z_1,r_0)\cap H_{z_1,e}^C$, $y_2\in B^d(z_2,r_0)\cap H_{z_2,e}^C$ and $y_3\in B^d(z_3,r_0)\cap H_{z_3,e}$.
Moreover, we choose constants $\tilde{c}\in(0,\infty)$ and $s\in(0,1)$ such that
\begin{equation}\label{eq:ExpectationDifferencePoints}
-\frac{c}{2} \exp(-4^d \kappa_d s^d)+ \Big({\sqrt{3}\over 2}+1\Big)^a (1-\exp(-4^d \kappa_d s^d)) \leq - \tilde{c}\,.
\end{equation}
In the following let $x\in W_0$ and $y_1,y_2,y_3$ be as above. If $\eta(B^d(x,4s))=0$, we have that
\begin{align*}
& \widehat{\mathcal{L}}^{(a)}_W\Big(\operatorname{DSF}(\eta+\delta_x+\sum_{i=1}^3\delta_{x+s y_i})\Big) - \widehat{\mathcal{L}}^{(a)}_W\Big(\operatorname{DSF}(\eta+\delta_{x}+\sum_{i=1}^2\delta_{x+s y_i})\Big)\\
& \leq s^a \Big(\|y_1-y_3\|^a+\|y_2-y_3\|^a+\|y_3\|^a-\ell_e(y_1,\delta_{0}+\delta_{y_1}+\delta_{y_2})^a -\ell_e(y_2,\delta_{0}+\delta_{y_1}+\delta_{y_2})^a\Big)\\
& \leq - \frac{c s^a}{2}\,.
\end{align*}
On the other hand, it holds that 
\begin{align*}
& \widehat{\mathcal{L}}^{(a)}_W\Big(\operatorname{DSF}(\eta+\delta_x+\sum_{i=1}^3 \delta_{x+sy_i})\Big) - \widehat{\mathcal{L}}^{(a)}_W\Big(\operatorname{DSF}(\eta+\delta_x+\sum_{i=1}^2\delta_{x+sy_i})\Big)\\
&\qquad \leq \ell_e(x+sy_3,\eta+\delta_x+\sum_{i=1}^3 \delta_{x+sy_i})\\
&\qquad \leq s^a \Big({\sqrt{3}\over 2}+1\Big)^a\, .
\end{align*}
Denote by $A$ the event that $\eta(B^d(x,4s))=0$. Then, combining the previous two inequalities with the fact that $\mathbb{P}(A)=\exp(-4^d\kappa_ds^d)$ and \eqref{eq:ExpectationDifferencePoints}, we see that
\begin{align*}
&\BE\Big[ \widehat{\mathcal{L}}^{(a)}_W\Big(\operatorname{DSF}(\eta+\delta_x+\sum_{i=1}^3 \delta_{x+y_i})\Big) - \widehat{\mathcal{L}}^{(a)}_W\Big(\operatorname{DSF}(\eta+\delta_x+\sum_{i=1}^2 \delta_{x+y_i})\Big)\Big] \\
&=\BE\Big[ \Big(\widehat{\mathcal{L}}^{(a)}_W\Big(\operatorname{DSF}(\eta+\delta_x+\sum_{i=1}^3 \delta_{x+y_i})\Big) - \widehat{\mathcal{L}}^{(a)}_W\Big(\operatorname{DSF}(\eta+\delta_x+\sum_{i=1}^2 \delta_{x+y_i})\Big)\Big){\bf 1}_A\\
&\qquad +\Big(\widehat{\mathcal{L}}^{(a)}_W\Big(\operatorname{DSF}(\eta+\delta_x+\sum_{i=1}^3 \delta_{x+y_i})\Big) - \widehat{\mathcal{L}}^{(a)}_W\Big(\operatorname{DSF}(\eta+\delta_x+\sum_{i=1}^2 \delta_{x+y_i})\Big)\Big){\bf 1}_{A^C}\Big]\\
&\leq -{\frac{cs^a}{2}}\,\exp(-4^d\kappa_ds^d)+s^a\Big({\sqrt{3}\over 2}+1\Big)^a\,(1-\exp(-4^d\kappa_ds^d))\\
&\leq -\tilde{c}s^a
\end{align*}
for $x\in W_0$ and $y_1\in B^d(sz_1,sr_0)\cap H_{sz_1,e}^C$, $y_2\in B^d(sz_2,sr_0)\cap H_{sz_2,e}^C$, $y_3\in B^d(sz_3,sr_0)\cap H_{sz_3,e}$. Now, Proposition \ref{prop:LPSvariance} concludes the proof for $a\geq 1$.

\begin{figure}[t]
\includegraphics[width=0.45\columnwidth]{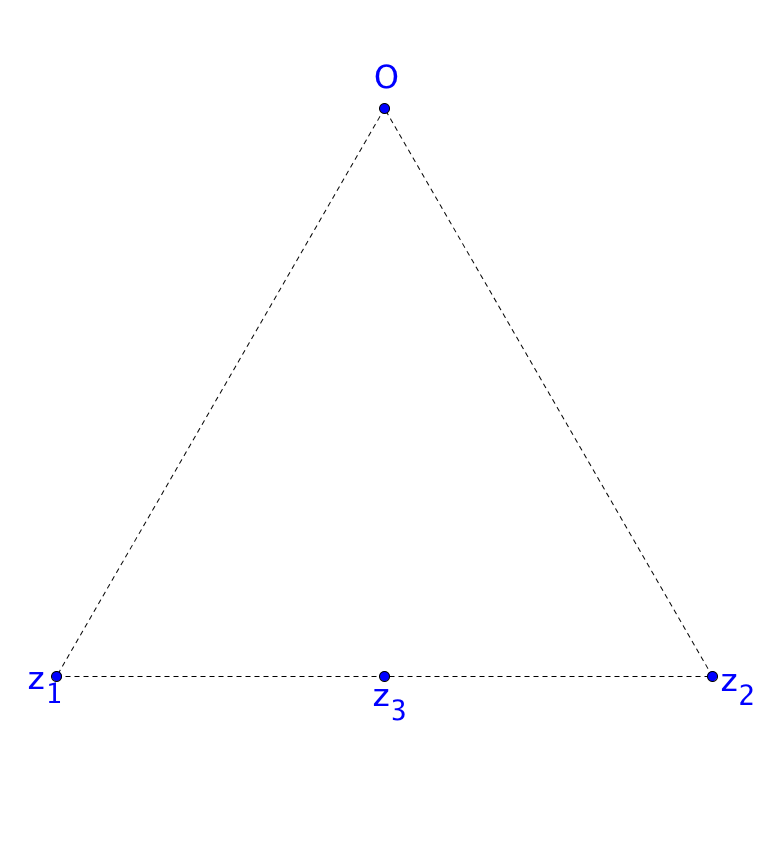}\quad
\includegraphics[width=0.45\columnwidth]{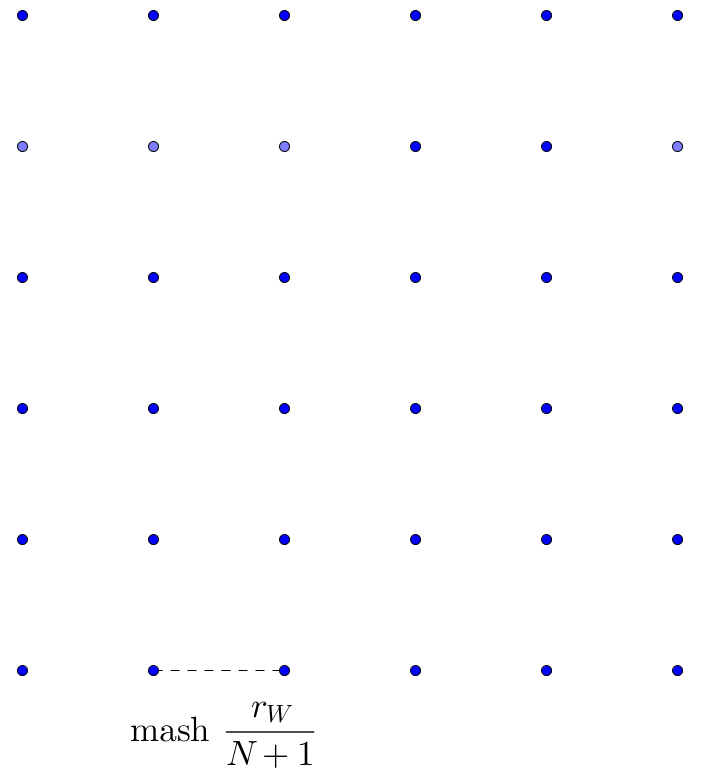}
\caption{Illustration of the constructions in the proof of Lemma \ref{lem:VarianceStrictlyPositive} for the planar case if $a\geq 1$ (left) and $0< a<1$ (right).}
\label{fig:Variance}
\end{figure}

Next, we consider the case that $0< a <1$. Let $N\in\N$ be a sufficiently large integer (a lower bound for $N$ will be stated later). We denote the points of the grid
$$
G:=\Big\{ \frac{r_W}{N+1}(i_1,\hdots,i_d)-(0,r_W,0,\hdots,0): i_1,\hdots,i_d=1,\hdots,N\Big\}
$$
by $z_1,\hdots,z_{N^d}$, see Figure \ref{fig:Variance}. Let $x\in W_0$ and $y_i\in B^d(z_i,r_W/(4N+4))$ for $i=1,\hdots,N^d$. By construction $y_1,\hdots,y_{N^d}$ are included in the $d$-dimensional hypercube $Q:=[0,r_W]\times [-r_W,0]\times [0,r_W]\times\hdots\times [0,r_W]$ so that $x+y_i\in W$ for all $i\in\{1,\hdots,N^d\}$. In the following, we derive a lower bound for
$$
\Big| \BE\Big[\widehat{\mathcal{L}}_W^{(a)}\Big(\operatorname{DSF}(\eta+\delta_x+\sum_{i=1}^{N^d}\delta_{x+y_i})\Big) - \widehat{\mathcal{L}}_W^{(a)}\Big(\operatorname{DSF}(\eta+\delta_x)\Big) \Big] \Big|\,.
$$
Adding the points $x+y_1,\hdots,x+y_{N^d}$ to the point process $\eta+\delta_x$ generates new edges. If $\eta(x+Q)=0$, then
$$
\ell_e(x+y_i,\eta+\delta_x+\sum_{i=1}^{N^d}\delta_{x+y_i})\geq \frac{r_W}{2(N+1)}\,, \qquad i=1,\hdots,N^d\,,
$$
where the lower bound is just half of the mash of the grid $G$. Consequently the expectation of the contribution of the new points is at least
$$
\exp(-r_W^d) N^d \frac{r_W^a}{2^a(N+1)^a}\,.
$$
On the other hand, inserting the additional points $x+y_1,\hdots,x+y_{N^d}$ can also shorten the edges belonging to some points of $\eta\cap W$, whereas the edge associated with $x$ is not affected. By the triangle inequality and the assumption $0\leq a<1$, we see that, for $z\in\eta\cap W$ with $\hat{n}(z,\eta+\delta_x+\sum_{i=1}^{N^d} \delta_{x+y_i})\in\{x+y_1,\hdots,x+y_{N^d}\}$,
\begin{align*}
\ell_e(z,\eta+\delta_x+\sum_{i=1}^{N^d}\delta_{x+y_i})^a -\ell_e(z,\eta+\delta_x)^a
& \geq \Big\|z-\hat{n}(z,\eta+\delta_x+\sum_{i=1}^{N^d}\delta_{x+y_i})\Big\|^a-\|z-x\|^a\\
& \geq -\Big\|\hat{n}(z,\eta+\delta_x+\sum_{i=1}^{N^d}\delta_{x+y_i})-x\Big\|^a\,,
\end{align*}
where we have used that $x\in H_{z,e}$ and that $\|z-x\|\geq \|z-\hat{n}(z,\eta+\delta_x+\sum_{i=1}^{N^d} \delta_{x+y_i})\|$. Since $\|x+y_i-x\|\leq \sqrt{d}\, r_W$ for $i=1,\hdots,N^d$, this means that
$$
0\geq \ell_e(z,\eta+\delta_x+\sum_{i=1}^{N^d}\delta_{x+y_i})^a-\ell_e(z,\eta+\delta_x)^a \geq -(\sqrt{d}\,r_W)^a\,, \quad z\in\eta+\delta_x\,.
$$
So, it remains to bound the expectation of the number of points of $\eta\cap W$ whose edges are affected by adding $x+y_1,\hdots,x+y_{N^d}$, i.e.,
$$
R:=\BE\Big[\sum_{z\in \eta\cap W} \I\{\ell_e(z,\eta+\delta_x+\sum_{i=1}^{N^d}\delta_{x+y_i}) < \ell_e(z,\eta+\delta_x) \}\Big]\,.
$$
Using the multivariate Mecke formula \eqref{eq:MeckeMulti} and denoting by $B_Q$ the smallest ball circumscribing the hypercube $Q$, we obtain that
$$
R \leq \kappa_d (\sqrt{d}\,r_W/2)^d + \int_{B_Q^C} \BP(\ell_e(z,\eta+\delta_z)\geq \inf_{u\in B_Q}\|u-z\|) \, \dint z\,.
$$
Transformation into spherical coordinates yields that
$$
R \leq \kappa_d (\sqrt{d}r_W/2)^d + d\kappa_d \int_{\sqrt{d}r_W/2}^\infty \exp(-\kappa_d (r-\sqrt{d}\,r_W/2)^d/2) \, r^{d-1} \, \dint r=: C_R\,.
$$
Consequently, we have that
\begin{align*}
& \BE\Big[\widehat{\mathcal{L}}^{(a)}_W\Big(\operatorname{DSF}(\eta+\delta_x+\sum_{i=1}^{N^d}\delta_{x+y_i})\Big) - \widehat{\mathcal{L}}_W^{(a)}\Big(\operatorname{DSF}(\eta+\delta_x)\Big)  \Big]\\
 &\qquad \geq \exp(-r_W^d) N^d  \frac{r_W^a}{2^a(N+1)^a} - C_R (\sqrt{d}\,r_W)^a\,.
\end{align*}
We now choose $N$ such that the right-hand side is larger than $1$ (note that this is always possible). Summarizing, we have shown that
$$
 \BE\Big[\widehat{\mathcal{L}}^{(a)}_W\Big(\operatorname{DSF}(\eta+\delta_x+\sum_{i=1}^{N^d}\delta_{x+y_i})\Big) - \widehat{\mathcal{L}}_W^{(a)}\Big(\operatorname{DSF}(\eta+\delta_x)\Big)  \Big]\geq 1
$$
for $x\in W_0$ and $y_i\in B^d(z_i,r_W/(4N+4))$, $i=1,\hdots,N^d$. Now, the assertion for $0< a<1$ follows from Proposition \ref{prop:LPSvariance}.

It remains to transfer the result to the variances of the edge-length functionals $\cL_t^{(a)}$ of the radial spanning tree. Since there is a constant $r_0\in(0,\infty)$ such that $r_{B^d(0,r)}=1$ for all $r\geq r_0$, we have that $\BV[\widehat{\mathcal{L}}^{(a)}_{B^d(0,r)}]\geq \hat{v}_a \kappa_dr^d$ with the same constant $\hat{v}_a\in(0,\infty)$ for all $r\geq r_0$. On the other hand, Lemma \ref{lem:VarianceDSF} implies that, for $a\geq 0$,
$$
v_a=\lim_{r\to\infty}\frac{\BV[\widehat{\mathcal{L}}^{(a)}_{B^d(0,r)}]}{\kappa_d r^d} \geq \hat{v}_a>0\,.
$$
This completes the proof.
\end{proof}

\section{Proof of Theorem \ref{thm:CLT}} \label{sec:ProofCLT}

Our aim is to apply Proposition \ref{prop:LPS}. In view of the variance asymptotics provided in Theorem \ref{thm:Variance}, it remains to investigate the first and the second-order difference operator of $t^{a/d}\cL_t^{(a)}$.

\begin{lemma}\label{lem:FirstOrderDifferenceOperator}
For any $a\geq 0$ there are constants $C_{1},C_{2}\in(0,\infty)$ only depending on $W$, $a$ and $d$ such that
$$
\BE[|D_z(t^{a/d}\cL_t^{(a)})|^5]\leq C_{1} \qquad\text{and}\qquad\BE[|D_{z_1,z_2}^2(t^{a/d}\cL_t^{(a)})|^5]\leq C_{2}
$$
for $z,z_1,z_2\in W$ and $t\geq 1$.
\end{lemma}
\begin{proof}
If $a=0$, we have $\cL_t^{(a)}=\eta_t(W)$ so that the statements are obviously true with $C_1=C_2=1$, for example. For $a> 0$, $\xi\in\bN$ and $z\in W$ we have that
$$
|D_z\cL^{(a)}(\RST(\xi))|\leq \ell(z,\xi+\delta_z)^a+\sum_{x\in\xi}\I\{\ell(x,\xi)\geq\|z-x\|\} \ \max_{\stackrel{x\in\xi}{\ell(x,\xi)\geq\|z-x\|}} \ell(x,\xi)^a\,.
$$
Consequently, it follows from Jensen's inequality that
\begin{align*}
\BE[|D_z(t^{a/d}\cL_t^{(a)})|^5]&\leq 16 t^{5a/d}\,\BE[\ell(z,\eta_t+\delta_z)^{5a}]\\
& \quad +16\BE\Big[\Big(\sum_{x\in\eta_t}\I\{\ell(x,\eta_t)\geq\|z-x\|\}\Big)^5\Big(\max_{x\in\eta_t\atop\ell(x,\eta_t)\geq\|z-x\|} t^{a/d}\ell(x,\eta_t)^a\Big)^5\Big]\,.
\end{align*}
Lemma \ref{lem:ExpTail} implies that $t^{5a/d}\BE[\ell(z,\eta_t+\delta_z)^{5a}]\leq c_{5a}$ for $z\in W$ and $t\geq 1$. For the second term, the Cauchy-Schwarz inequality yields the bound
\begin{align}\label{eq:10thMomentEstimate}
\Big(\BE\Big[\Big(\sum_{x\in\eta_t}\I\{\ell(x,\eta_t)\geq\|z-x\|\}\Big)^{10}\Big]\,\Big)^{1/2}\;\Big(\BE[\max_{x\in\eta_t\atop\ell(x,\eta_t)\geq\|z-x\|}t^{10a/d}\ell(x,\eta_t)^{10a}]\,\Big)^{1/2} \,.
\end{align}
The $10$th moment of $\sum_{x\in\eta_t}\I\{\ell(x,\eta_t)\geq\|z-x\|\}$ can be expressed as a linear combination of terms of the form
$$
\BE\Big[\sum_{(x_1,\ldots,x_k)\in\eta_{t,\neq}^k}\I\{\ell(x_i,\eta_t)\geq\|x_i-z\|,\,i=1,\ldots,k\}\Big]
$$
with $k\in\{1,\ldots,10\}$.
Applying the multivariate Mecke formula \eqref{eq:MeckeMulti}, the monotonicity relation \eqref{eq:monotone}, H\"older's inequality and Lemma \ref{lem:ExpTail} yields, for each such $k$,
\begin{align*}
&\BE\Big[\sum_{(x_1,\ldots,x_k)\in\eta_{t,\neq}^k}\I\{\ell(x_i,\eta_t)\geq\|x_i-z\|,\,i=1,\ldots,k\}\Big]\\
&=t^k\int_{W^k}\BP\Big(\ell(x_i,\eta_t+\sum_{i=1}^k\delta_{x_i})\geq\|x_i-z\|\,,i=1,\ldots,k\Big)\,\dint(x_1,\ldots,x_k)\\
&\leq t^k\int_{W^k}\BP\big(\ell(x_i,\eta_t+\delta_{x_i})\geq\|x_i-z\|\,,i=1,\ldots,k\big)\,\dint(x_1,\ldots,x_k)\\
&\leq t^k\int_{W^k}\prod_{i=1}^k\BP\big(\ell(x_i,\eta_t+\delta_{x_i})\geq\|x_i-z\|\big)^{1/k}\,\dint(x_1,\ldots,x_k)\\
&\leq \Big(t\int_W \exp(-t\alpha_W\kappa_d\|x-z\|^d/k) \,\dint x\Big)^k\,.
\end{align*}
Introducing spherical coordinates and replacing $W$ by $\R^d$, we see that
\begin{align*}
& t\int_{W} \exp(-t\alpha_W\kappa_d\|x-z\|^d/k) \, \dint x \leq td\kappa_d\int_0^\infty \exp(-t\alpha_W\kappa_dr^d/k) \, r^{d-1}\,\dint r = \frac{k}{\alpha_W}\,.
\end{align*}
This implies that $\BE\big[\big(\sum_{x\in\eta_t}\I\{\ell(x,\eta_t)\geq\|z-x\|\}\big)^{10}\big]$ is uniformly bounded for $z\in W$ and $t\geq 1$. Moreover, for $u\geq 0$ we have that
\begin{align*}
& \BP\big(t^{a/d}\max_{\stackrel{x\in\eta_t}{\ell(x,\eta_t)\geq\|z-x\|}} \ell(x,\eta_t)^a\geq u\big)\\
& = \BP(\exists x\in\eta_t: \ell(x,\eta_t)\geq t^{-1/d}u^{1/a}\,,\ell(x,\eta_t)\geq \|z-x\|)\\
& \leq \BE\Big[\sum_{x\in\eta_t}\I\{\ell(x,\eta_t)\geq\max\{t^{-1/d}u^{1/a},\|z-x\|\}\}\Big] \allowdisplaybreaks\\
& = t \int_{W} \BP\big(\ell(x,\eta_t+\delta_x)\geq\max\{t^{-1/d}u^{1/a},\|z-x\|\}\big) \, \dint x\\
& \leq t\int_{\R^d} \exp(-t\alpha_W\kappa_d\max\{t^{-1/d}u^{1/a},\|z-x\|\}^d) \, \dint x\\
& = \kappa_du^{d/a} \exp(-\alpha_W\kappa_du^{d/a}) + t\int_{\R^d\setminus B^d(z,t^{-1/d}u^{1/a})} \exp(-t\alpha_W\kappa_d\|z-x\|^d) \, \dint x \allowdisplaybreaks\\
& = \kappa_du^{d/a} \exp(-\alpha_W\kappa_du^{d/a}) + d\kappa_dt \int_{t^{-1/d}u^{1/a}}^\infty \exp(-t\alpha_W\kappa_d r^d) \, r^{d-1} \, \dint r\\
& = \kappa_du^{d/a} \exp(-\alpha_W\kappa_du^{d/a}) + \frac{1}{\alpha_W} \exp(-\alpha_W\kappa_d u^{d/a}) \,,
\end{align*}
where we have used the Mecke formula \eqref{eq:MeckeMulti}, Lemma \ref{lem:ExpTail} and a transformation into spherical coordinates. This exponential tail behaviour implies that the second factor of the product in \eqref{eq:10thMomentEstimate} is uniformly bounded for all $t\geq 1$. Altogether, we see that $\BE[|D_z(t^{a/d}\cL_t^{(a)})|^5]\leq C_{1}$ for all $z\in W$ and $t\geq 1$ with a constant $C_{1}\in(0,\infty)$ only depending on $W$, $a$ and $d$.

For the second-order difference operator we have that, for $z_1,z_2\in W$,
\begin{equation}\label{eq:DecompositionSecondDifferenceOperatorMoment}
\begin{split}
\BE[|D_{z_1,z_2}^2(t^{a/d}\cL_t^{(a)})|^5] &\leq 16 t^{5a/d}\,\BE[|D_{z_1}\cL^{(a)}(\operatorname{RST}(\eta_t))|^5]\\
&\qquad +16 t^{5a/d}\,\BE[|D_{z_1}\cL^{(a)}(\operatorname{RST}(\eta_t+\delta_{z_2}))|^5]\,.
\end{split}
\end{equation}
In view of the argument for the first-order difference operator above, it only remains to consider the second term. For $z_1,z_2\in W$ and $\xi\in\bN$ we have that
\begin{align*}
&\big|D_{z_1}(t^{a/d}\cL^{(a)}(\operatorname{RST}(\xi+\delta_{z_2})))\big| \leq t^{a/d}\,\ell(z_1,\xi+\delta_{z_1}+\delta_{z_2})^a\\
&\qquad+\Big(\sum_{x\in\xi+\delta_{z_2}}\I\{\ell(x,\xi+\delta_{z_2})\geq\|z_1-x\|\}\Big)\Big(\max_{\stackrel{x\in\xi+\delta_{z_2}}{\ell(x,\xi+\delta_{z_2})\geq\|z_1-x\|}}t^{a/d}\ell(x,\xi+\delta_{z_2})^a\Big)\,.
\end{align*}
Using the monotonicity relation \eqref{eq:monotone}, we find that
\begin{align*}
t^{a/d}\,\ell(z_1,\xi+\delta_{z_1}+\delta_{z_2})^a &\leq t^{a/d}\,\ell(z_1,\xi+\delta_{z_1})^a\,,\\
\sum_{x\in\xi+\delta_{z_2}}\I\{\ell(x,\xi+\delta_{z_2})\geq\|z_1-x\|\} &\leq 1+\sum_{x\in\xi}\I\{\ell(x,\xi)\geq\|z_1-x\|\}
\end{align*}
and
$$\max_{\stackrel{x\in\xi+\delta_{z_2}}{\ell(x,\xi+\delta_{z_2})\geq\|z_1-x\|}}t^{a/d}\,\ell(x,\xi+\delta_{z_2})^a \leq t^{a/d}\,\ell(z_2,\xi+\delta_{z_2})^a+\max_{\stackrel{x\in\xi}{\ell(x,\xi)\geq\|z_1-x\|}}t^{a/d}\,\ell(x,\xi)^a\,.$$
This implies that for $\xi=\eta_t$ the $5$th and the $10$th moment of these expressions are uniformly bounded in $t$ by the same arguments as above. In view of \eqref{eq:DecompositionSecondDifferenceOperatorMoment} this shows that $\BE[|D_{z_1,z_2}^2(t^{a/d}\cL_t^{(a)})|^5]\leq C_{2}$ for all $z_1,z_2\in W$ and $t\geq 1$ with a constant $C_{2}\in(0,\infty)$ only depending on $W$, $a$ and $d$.
\end{proof}

While Lemma \ref{lem:FirstOrderDifferenceOperator} shows that assumption \eqref{eq:MomentsDifferenceOperator} in Proposition \ref{prop:LPS} is satisfied, the following result puts us in the position to verify also assumption \eqref{eq:SupDifferenceOperator} there.

\begin{lemma}\label{lem:SecondOrderDifferenceOperator}
For $a\geq 0$, $z_1,z_2\in W$ and $t\geq 1$,
$$
\BP(D_{z_1,z_2}^2\cL_t^{(a)}\neq 0) \leq (2+2/\alpha_W)\,\exp(-t\alpha_W\kappa_d\|z_1-z_2\|^d/2^{d})
$$
with $\alpha_W$ as in Lemma \ref{lem:KonstanteAlphad}.
\end{lemma}

\begin{proof}
Since $D^2_{z_1,z_2}\cL_t^{(a)}=0$ for $a=0$, we can restrict ourselves to $a>0$ in the following. First notice that
$$
D_{z_1,z_2}^2\cL_t^{(a)} = \sum_{y\in\eta_t}D_{z_1,z_2}^2 \big(\ell(y,\eta_t)^a\big) + D_{z_1}\big(\ell(z_2,\eta_t+\delta_{z_2})^a\big) + D_{z_2}\big(\ell(z_1,\eta_t+\delta_{z_1})^a\big) \,.
$$
To have $D_{z_1,z_2}^2\cL_t^{(a)}\neq 0$, at least one of the above terms has to be non-zero. Moreover, for $x,z\in W$ and $\xi\in\bN$, $D_z(\ell(x,\xi+\delta_x)^a)\neq 0$ and $D_z\ell(x,\xi+\delta_x)\neq 0$ are equivalent. Thus,
\begin{align*}
\BP(D_{z_1,z_2}^2\cL_t^{(a)}\neq 0) & \leq \BP(\exists y\in\eta_t: D^2_{z_1,z_2}\big(\ell(y,\eta_t)^a\big) \neq 0)\\
& \qquad +\BP(D_{z_1}\big(\ell(z_2,\eta_t+\delta_{z_2})^a\big)\neq 0)+\BP(D_{z_2}\big(\ell(z_1,\eta_t+\delta_{z_1})^a\big)\neq 0)\\
& = \BP(\exists y\in\eta_t: D^2_{z_1,z_2}\ell(y,\eta_t) \neq 0)\\
& \qquad +\BP(D_{z_1}\ell(z_2,\eta_t+\delta_{z_2})\neq 0)+\BP(D_{z_2}\ell(z_1,\eta_t+\delta_{z_1})\neq 0)\,.
\end{align*}
Since $D_{z_1}\ell(z_2,\eta_t+\delta_{z_2})\neq 0$ requires that $\ell(z_2,\eta_t+\delta_{z_2})\geq \|z_1-z_2\|$, application of Lemma \ref{lem:ExpTail} yields
$$
\BP(D_{z_1}\ell(z_2,\eta_t+\delta_{z_2})\neq 0) \leq \BP(\ell(z_2,\eta_t+\delta_{z_2})\geq \|z_1-z_2\|) \leq \exp(-t\alpha_W\kappa_d\|z_1-z_2\|^d)
$$
and similarly
$$
\BP(D_{z_2}\ell(z_1,\eta_t+\delta_{z_1})\neq 0) \leq \exp(-t\alpha_W\kappa_d\|z_1-z_2\|^d)\,.
$$
Using Mecke's formula \eqref{eq:MeckeMulti}, the fact that $D^2_{z_1,z_2}\ell(y,\eta_t+\delta_y)\neq 0$ implies $\ell(y,\eta_t+\delta_y)\geq \max\{\|y-z_1\|,\|y-z_2\|\}$ and once again Lemma \ref{lem:ExpTail}, we finally conclude that
\begin{align*}
 \BP(\exists y\in\eta_t: D^2_{z_1,z_2}\ell(y,\eta_t)\neq 0)
& \leq \BE\Big[\sum_{y\in\eta_t} \I\{ D^2_{z_1,z_2}\ell(y,\eta_t)\neq 0 \}\Big]\\
& = t\int_W \BP(D^2_{z_1,z_2}\ell(y,\eta_t+\delta_y)\neq 0) \, \dint y\\
& \leq t\int_W \BP(\ell(y,\eta_t+\delta_y)\geq \max\{\|y-z_1\|,\|y-z_2\|\}) \, \dint y \\
&\leq t\int_W \exp(-t\alpha_W\kappa_d\max\{\|y-z_1\|,\|y-z_2\|\}^d) \, \dint y \allowdisplaybreaks\\
&\leq t\int_{\R^d\setminus B^d(z_1,\|z_1-z_2\|/2)}\exp(-t\alpha_W\kappa_d\|y-z_1\|^d) \,\dint y\\
&\hspace{1.5cm}+t\int_{\R^d\setminus B^d(z_2,\|z_1-z_2\|/2)} \exp(-t\alpha_W\kappa_d\|y-z_2\|^d) \,\dint y\\
&\leq \frac{2}{\alpha_W} \exp(-t\alpha_W\kappa_d\|z_1-z_2\|^d/2^d)\,.
\end{align*}
Consequently,
$$
\BP(D_{z_1,z_2}^2\cL_t^{(a)}\neq 0) \leq (2+2/\alpha_W) \, \exp(-t\alpha_W\kappa_d\|z_1-z_2\|^d/2^{d})\,,
$$
which proves the claim.
\end{proof}

After these preparations, we can now use Proposition \ref{prop:LPS} to prove Theorem \ref{thm:CLT}.

\begin{proof}[Proof of Theorem \ref{thm:CLT}]
It follows from Lemma \ref{lem:SecondOrderDifferenceOperator} by using spherical coordinates that
\begin{align*}
& \sup_{z_1\in W,\, t\geq 1}t\int_{W}\BP(D_{z_1,z_2}^2(t^{a/d}\cL_t^{(a)})\neq 0)^{1/20}\,\dint z_2\\
& \leq (2 + 2/\alpha_W)\sup_{z_1\in W,\, t\geq 1}t\int_{\R^d} \exp\big(-t\alpha_W\kappa_d\|z_1-z_2\|^d/(20\cdot 2^{d})\big) \,\dint z_2\\
& \leq (2 + 2/\alpha_W) \sup_{t\geq 1}\,t\int_{\R^d} \exp\big(-t\alpha_W\kappa_d\|z\|^d/(20\cdot2^d)\big) \,\dint z\,, \\
&= 20\cdot 2^d (2/\alpha_W+2/\alpha_W^2)\,,
\end{align*}
which is finite. Moreover, Lemma \ref{lem:FirstOrderDifferenceOperator} shows that there are constants $C_{1},C_{2}\in(0,\infty)$ only depending on $W$, $a$ and $d$ such that
$$
\BE[|D_z(t^{a/d}\cL_t^{(a)})|^5]\leq C_{1}\qquad\text{and}\qquad\BE[|D_{z_1,z_2}^2(t^{a/d}\cL_t^{(a)})|^5]\leq C_{2}
$$
for $z,z_1,z_2\in W$ and $t\geq 1$. Finally, Lemma \ref{lem:ExistenceVarianceAsymptotics} and Lemma \ref{lem:VarianceStrictlyPositive} ensure the existence of constants $v_a, t_0\in(0,\infty)$ depending on $W$, $a$ and $d$ such that $t^{-1}\BV[t^{a/d}\cL_t^{(a)}]\geq v_a/2$ for all $t\geq t_0$. Consequently, an application of Proposition \ref{prop:LPS} completes the proof of Theorem \ref{thm:CLT}.
\end{proof}

\subsection*{Acknowledgements}
This research was supported through the program ``Research in Pairs'' by the Mathematisches Forschungsinstitut Oberwolfach in 2014. MS has been funded by the German Research Foundation (DFG) through the research unit ``Geometry and Physics of Spatial Random Systems'' under the grant HU 1874/3-1. CT has been supported by the German Research Foundation (DFG) via SFB-TR 12.

\end{document}